\documentclass[12pt]{amsart}
\usepackage[english]{babel}
\usepackage[utf8]{inputenc}
\usepackage{lipsum}
\usepackage{stix}
\usepackage{float}
\usepackage{fullpage}
\usepackage{bbm}
\usepackage{dsfont}
\usepackage{relsize}
\usepackage{amsthm}
\usepackage[breaklinks]{hyperref}
\usepackage{blindtext}
\usepackage{amsmath,amssymb}
\usepackage{mathtools}
\usepackage{autonum}
\usepackage{enumitem}
\usepackage[usenames,dvipsnames,svgnames]{xcolor}
\usepackage{mathrsfs}
\usepackage{csquotes}
\usepackage[]{hyperref}
\hypersetup{colorlinks=true,linkcolor=WildStrawberry,citecolor=cyan}
\newtheorem{thm}{Theorem}[section]

\newtheorem{lem}[thm]{Lemma}

\theoremstyle{definition}
\newtheorem{defn}[thm]{Definition}
\theoremstyle{remark}
\newtheorem{rem}[thm]{Remark}

\numberwithin{equation}{section}

\newcommand{\To}{\longrightarrow}

\makeatletter
\renewcommand\paragraph{\@startsection{paragraph}{4}{\z@}%
	{-2.5ex\@plus -1ex \@minus -.25ex}%
	{1.25ex \@plus .25ex}%
	{\normalfont\normalsize\centering\bfseries}}
\newcommand{\nocontentsline}[3]{}
\let\origcontentsline\addcontentsline
\newcommand\stoptoc{\let\addcontentsline\nocontentsline}
\newcommand\resumetoc{\let\addcontentsline\origcontentsline}
\makeatother
\begin{document}
	\title
	[Semilinear parabolic equation involving sub-Laplacian]
	{Inhomogeneous Parabolic Equations with Hardy Potential and Memory on the Heisenberg Group}
	
	\author[V.\, Kumar,\,\, P.\,Oza,\,\, D.\,Suragan]
	{Vishvesh Kumar,\,\, Priyank Oza,\,\,Durvudkhan Suragan }
	\address{Vishvesh\,Kumar \hfill\break
		Department of Mathematical Sciences\\
		Indian Institute of Technology (BHU)  \newline
		Varanasi, Uttar Pradesh, 221005, India.}
	\email{ vishveshmishra@gmail.com, vishvesh.mat@iitbhu.ac.in}
    
	\address{Priyank\,Oza \hfill\break
		Department of General Sciences \\
		Birla Institute of Technology and Science Pilani, Dubai Campus \newline
		International Academic City,
        Dubai, 345055, UAE.}
	\email{priyank.oza3@gmail.com, priyankkumar.oza@dubai.bits-pilani.ac.in}
    
	\address{Durvudkhan\,Suragan \hfill\break
		Department of Mathametics\\
		Nazarbayev University \newline
		53 Kabanbay Batyr Ave, Astana, Kazakhstan-010000.}
	\email{durvudkhan.suragan@nu.edu.kz}
    
	\thanks{Submitted \today.  Published-----.}
	\subjclass[2020]{35A01, 35H20, 35R03, 35K58, 35B44}
	\keywords{Partial differential equations on the Heisenberg group, inhomogeneous parabolic equation, Hardy potential, local existence, finite time blow-up, lifespan estimates}
	\maketitle
	\begin{abstract}
		We study a class of inhomogeneous parabolic equations on the Heisenberg group $\mathbbm{H}^N$ with Hardy-type singular potentials, nonlocal memory terms,  and a space-time forcing term:
		\begin{align}
			\partial_tu-\Delta_{H}u=\lambda \frac{\psi u}{\|\cdot\|^{2}_{H}}+\frac{1}{\Gamma(\gamma)}\int_0^t(t-\tau)^{\gamma-1}|u(\tau)|^{p}d\tau+t^\alpha f \text{ in } \,\mathbbm{H}^N\times (0,T).
		\end{align}
		Here, $\gamma\in [0,1),$ $\alpha\in (-1,\infty),$ $p>1,$ $\lambda>0,$ and $\psi(\cdot)=|\nabla_{H}\|\cdot\|_{H}|^2,$ where $\nabla_H$ is the horizontal gradient  associated to $\Delta_H.$ Also, $\|\cdot\|_{H}$ and $\Delta_{H}$ denote the Kor\'anyi norm and sub-Laplacian associated with the sub-Riemannian geometry of $\mathbbm{H}^N,$ respectively. The combination of a singular Hardy potential and a memory kernel introduces significant analytical challenges. Using a Harnack-type inequality adapted to the Heisenberg group setting, we obtain quantitative positivity estimates that enable a detailed blow-up analysis. We identify parameter regimes depending on  $p,\gamma,\alpha$ leading to finite-time blow-up or instantaneous blow-up, and establish local well-posedness in the absence of the Hardy potential. These results reveal an interplay between the spatial singularity, temporal nonlocality and a time-dependent forcing term. Finally, under a suitable lower bound on the forcing term $f,$ we derive an explicit lifespan estimate for local-in-time solutions. 
	\end{abstract}	
	{\hypersetup{linkcolor=black}
		\tableofcontents}
	\section{Introduction}
	Parabolic equations with nonlocal memory and singular potentials have gained significant attention for modeling complex diffusion phenomena beyond the classical framework. Such formulations arise naturally in materials with thermal relaxation, viscoelasticity, or anomalous heat conduction. In particular, the inclusion of a memory term of Volterra type captures temporal nonlocality, describing processes in which the present rate of change depends on the cumulative history of the state variable. Concurrently, the Hardy potential introduces a spatially singular weight that reflects the influence of geometric constraints or concentration effects near the origin, thereby modifying the effective diffusion rate. These terms, in fact, affect the qualitative behaviour of solutions, such as finite-time blow-up, instantaneous blow-up and existence of local-in-time solutions. 
	
	The phenomenon of blow-up, signifying the loss of global existence in finite time, plays a central role in the qualitative theory of nonlinear evolution equations. As highlighted in the classical review by Levine \cite{Levine}, this blow-up describes the finite-time breakdown of solutions in various physical and engineering models, including fluid flow \cite{F Strau}, quantum mechanics \cite{Q Gold}, and chemical reactions \cite{C Frank}. 
	
	Motivated by these aspects, our work contributes to understanding how singular potentials and memory terms can affect the blow-up behaviour of solutions in the non-Euclidean setting of the Heisenberg group. In particular, we study a class of semilinear parabolic equations on the Heisenberg group of the form
	\begin{align}\label{eq 0.1}
		\partial_tu-\Delta_{H}u=\lambda \frac{\psi u}{\|\cdot\|^{2}_{H}}+\frac{1}{\Gamma(\gamma)}\int_0^t(t-\tau)^{\gamma-1}|u(\tau)|^{p}d\tau+t^\alpha f \text{ in } \,\mathbbm{H}^N\times (0,T),
	\end{align}
	where $\gamma\in [0,1),$ $\alpha\in (-1,\infty),$ $\lambda>0,$ and $\psi(\cdot)=|\nabla_{H}\|\cdot\|_{H}|^2.$  
	
	\noindent In comparison to the semilinear parabolic model studied by the second and third-named authors \cite{Oza Sur 2}
	\begin{align}\label{SuO}
		\partial_tu-\Delta_{H}u=|u|^p+t^\alpha f \text{ in } \,\mathbbm{H}^N\times (0,T),
	\end{align}
	the coexistence of a memory-dependent source term and a singular Hardy potential combines temporal nonlocality with spatial singularity in \eqref{eq 0.1}. In the limiting case, when $\lambda\To0$ and $\gamma\To 0,$ \eqref{eq 0.1} reduces to \eqref{SuO}. Therefore, the current model can be viewed as a nonlocal temporal extension of \eqref{SuO} together with a spatial singularity. This introduces significant analytical difficulties absent in the classical diffusion models.
	
	This work aims to characterize the finite-time blow-up regime of solutions, identifying the critical values of $p$ up to which global-in-time existence is not possible. To illustrate the influence of the model parameters, we summarize the blow-up regimes in Table \ref{table 1}.
	\begin{table}[H]
		\centering
		\begin{tabular}{|c|c|c|} 
			\hline
			Case &Blow-up regime &Corresponding result\\ [0.5ex] 
			\hline\hline
			
			\hspace{-.25cm} \quad$\bullet$ $\gamma\in [0,1),$\,$\alpha\in (-1,0)$  &\,\,\qquad$1<p<\frac{Q+2\gamma-\mu(\lambda)-2\alpha}{Q-2-\mu(\lambda)-2\alpha}$ &Theorem \ref{Thm Blowup O}\\
			\hspace{-3.5cm} \qquad\qquad\qquad\quad\,\,\quad$\bullet$ $\gamma\in [0,1) ,$\,$\alpha\in (0,\infty)$ & $p>1$ &Theorem \ref{Thm Blowup O2}\\[1mm]
			\hspace{-2.67cm}\qquad\qquad\quad\,\,\,\quad$\bullet$ $\gamma\in (0,1),$\,$\alpha\in [0,\infty)$ & $p>1$ &Theorem \ref{Thm Blowup O3}\\
			\hline
		\end{tabular}
		\caption{Blow-up regime.}
		\label{table 1}
	\end{table}
	To handle the singular potential, we rely on the Heisenberg group analogue of the classical Hardy inequality (Theorem \ref{Hardy}), which provides sharp control of the potential term. In fact, in view of it, for a fixed $\lambda\in \big(0,\big(\frac{Q-2}{2}\big)^2\big),$ we define 
	\begin{align}
		\mu=\mu(\lambda)\coloneqq\frac{Q-2}{2}-\sqrt{\left(\frac{Q-2}{2}\right)^2-\lambda}.
	\end{align}
	This, in particular, is a non-trivial root of the equation
	\begin{align}
		\mu^2-(Q-2)\mu+\lambda=0.
	\end{align}	
	Using the transformation $v(\xi,t)=\|\xi\|_H^{\mu(\lambda)}u(\xi,t)$, the main equation \eqref{eq 0.1} can be reformulated in a simpler form, facilitating the analysis of blow-up. Details of this transformation and the resulting equivalent system are given in Section \ref{Preliminaries}, see in particular, Subsection \ref{WT sub}.

	The Fujita phenomenon, originating from Fujita \cite{Fujita}, identifies the critical exponent separating global existence from blow-up for semilinear heat equations. Subsequent works by Hayakawa \cite{Haya}, Sugitani \cite{Sugitani}, and Kobayashi et al. \cite{Koba} refined the analysis by determining sharp thresholds and extending the results to broader classes of nonlinearities and initial data. Later, Bandle and Levine \cite{Bandle} considered nonhomogeneous sources, while Jleli et al. \cite{Jleli} addressed time-dependent coefficients, revealing the sensitivity of the critical behavior to temporal parameters.
	
	In the broader context of stratified Lie groups, Pascucci \cite{Pasu} extended classical Fujita-type results, and more recently, the third-named author and Talwar \cite{Talwar} studied semilinear subelliptic heat equations with a forcing term $f$ on arbitrary stratified Lie groups $\mathbb{G} = (\mathbb{R}^N, \star)$ governed by the sub-Laplacian
	\begin{align}
		\Delta_{\mathbbm{G}}\coloneqq-\sum_{i=1}^mX_i^2,
	\end{align}
	where $\{X_i\}_{i=1}^m$ generates $\mathbb{G}$. Furthermore, in \cite{Talwar2}, they studied the influence of a Hardy potential on the their model and \cite{kir} investigates a homogeneous parabolic equation on the Heisenberg with a memory term. We also mention the recent works of Loiudice \cite{Loiu, Loiu 2} on sub-elliptic problems with Hardy potentials on stratified Lie groups. Although their equations are stationary, these works illustrate the strong influence of Hardy-type singularities in the existence, non-existence, and local behaviour of solutions. Recently, several researchers have explored parabolic equations
	with sub-elliptic operators, see, for instance, \cite{Ahmad2, Ahmad, Berik, Jleli2, Poho, Ruzha2, Ruzha, ZhangQ1}. Related insights on the instantaneous blow-up results for heat equation with singular potentials appear in \cite{Gold, Gold 2}. In particular, on the Heisenberg group, equation \eqref{SuO} with $\alpha>-1$ has been very recently investigated in \cite{Oza Sur 2}, where the Fujita exponent is determined to be:
	\begin{align}\label{palpha}
		p^\alpha_F=\begin{cases}
			\frac{Q-2\alpha}{Q-2-2\alpha} &\text{ for }-1<\alpha<0,\\
			\infty &\text{ for } \alpha>0.	
		\end{cases}
	\end{align}
	A remarkable feature 
	of this exponent is its discontinuity at $\alpha=0.$ Indeed:
	\begin{align}
		\lim_{\alpha\To 0^{-}}p^\alpha_F=\frac{Q}{Q-2}=p_F^0, \text{ whereas } \lim_{\alpha\To 0^{+}}p^\alpha_F=\infty.
	\end{align}
	Notably, unlike in the Euclidean setting, the Fujita exponent \eqref{palpha} remains bounded across all dimensions when $\alpha\in (-1,0].$ Furthermore, the authors derived upper bounds on the lifespan of local-in-time
	solutions in the subcritical regime.  
	
	
	Inspired by these developments, we consider the equation \eqref{eq 0.1} to enrich the understanding of such models in the non-commutative and non-Euclidean contexts. Our main results-- Theorems \ref{Thm Blowup O}--\ref{Thm Blowup O3} identify the blow-up regime corresponding to Fujita-type phenomena, specifying the parameter ranges where solutions cannot exist globally.
	
	The novelty of this work lies in simultaneous presence of a singular Hardy potential, a nonlocal memory, and a time-dependent forcing. To the best of our knowledge, such a combination has not been previously studied. In particular, our results are new even in the Euclidean setting, including the simplified case $f=0$ in \eqref{eq 0.1}.

	
	We are now ready to formally state our main results concerning the blow-up behavior of solutions to problem \eqref{eq 0.1} irrespective of the non-negative initial data. 
	


\begin{thm} \label{Thm Blowup O} Let $\mathbb{H}^N$ be the Heisenberg group with homogeneous dimension $Q:=2N+2,$ and 
	let $\gamma\in [0,1)$ and $\alpha\in (-1,0).$ We assume that  
	\begin{align}
		 \int_{\mathbbm{H}^N}\|\xi\|_H^{-\mu(\lambda)}u_0(\xi)d\xi<\infty,\,0< \int_{\mathbbm{H}^N}\|\xi\|_H^{-\mu(\lambda)}f(\xi)d\xi<\infty.
	\end{align} 
	Then, for any $1<p<\frac{Q+2\gamma-\mu(\lambda)-2\alpha}{Q-2-\mu(\lambda)-2\alpha}$,  problem \eqref{eq 0.1} does not admit a global-in-time weak solution.
\end{thm}
Moreover, in the case when either $\alpha>0$ or $\gamma>0,$ the blow-up threshold increases substantially, rising from $\frac{Q+2\gamma-\mu(\lambda)-2\alpha}{Q-2-\mu(\lambda)-2\alpha}$ to $\infty.$ More precisely, we have the following results:

\begin{thm}\label{Thm Blowup O2} Let $\mathbb{H}^N$ be the Heisenberg group with homogeneous dimension $Q:=2N+2,$ and
	let $\gamma\in [0,1)$ and  $\alpha\in (0,\infty).$ We assume that 	\begin{align}
		\int_{\mathbbm{H}^N}\|\xi\|_H^{-\mu(\lambda)}u_0(\xi)d\xi<\infty,\,0< \int_{\mathbbm{H}^N}\|\xi\|_H^{-\mu(\lambda)}f(\xi)d\xi<\infty.
	\end{align}   
	Then, for any $p>1,$ problem \eqref{eq 0.1} does not admit a global-in-time weak solution.
\end{thm}

\begin{thm}\label{Thm Blowup O3} Let $\mathbb{H}^N$ be the Heisenberg group with homogeneous dimension $Q:=2N+2,$ and
	let $\gamma\in (0,1)$ and  $\alpha\in [0,\infty).$ We assume that 	\begin{align}
		\int_{\mathbbm{H}^N}\|\xi\|_H^{-\mu(\lambda)}u_0(\xi)d\xi<\infty,\,0< \int_{\mathbbm{H}^N}\|\xi\|_H^{-\mu(\lambda)}f(\xi)d\xi<\infty.
	\end{align}  
	Then, for any $p>1,$ problem \eqref{eq 0.1} does not admit a global-in-time weak solution.
\end{thm}

In Section \ref{Local wel}, a local well-posedness result is established in the absence of the Hardy potential (Theorem \ref{Local}), while in its presence, we show that solutions blow-up immediately once the exponent $p$ exceeds a critical value. We show that for certain parameter choices involving the Hardy potential or the memory term, the equation exhibits \textit{instantaneous blow-up}, meaning that no local-in-time weak solution exists for \eqref{eq 0.1} (see Theorems \ref{Thm Blowup O4} and \ref{Thm Blowup O5}).

\begin{table}[H]
	\centering
	\begin{tabular}{|c|c|c|} 
		\hline
		Case &Regime &Corresponding result\\ [0.5ex] 
		\hline\hline
		
		\hspace{-.25cm} \quad$\bullet$ $\gamma=0,$\,$\alpha\in (-1,\infty)$  &$p>\frac{\mu(\lambda)+2}{\mu(\lambda)}$ &Theorem \ref{Thm Blowup O4}\\
		\hspace{-3.5cm} \qquad\qquad\qquad\quad\,\,\quad$\bullet$ $\gamma>0,$\,$\alpha\in (-1,\infty)$ (under the assumption \ref{H1}) & $p>\frac{\mu(\lambda)+2}{\mu(\lambda)}$ &Theorem \ref{Thm Blowup O5}\\
		\hline
	\end{tabular}
	\caption{Instantaneous blow-up regime.}
	\label{table 2}
\end{table}


The organization of this paper is as follows. In Section \ref{Preliminaries}, we introduce preliminary definitions and the framework used throughout the paper. Section \ref{Local wel} is devoted to establishing results related to local existence and instantaneous blow-up of solutions. In Section \ref{Glo ex}, we provide the detailed proofs of the Theorems \ref{Thm Blowup O}--\ref{Thm Blowup O3}. In Section \ref{Life s}, we provide a life span estimate to local-in-time solutions.

Throughout this paper, the constant $C$ may vary from line to line, but its dependence on parameters will be clearly indicated when relevant.

\section{Preliminaries}\label{Preliminaries}
In this section, we recall the fundamental notions and differential operators associated with the Heisenberg group
$\mathbbm{H}^N.$ 

\subsection{Structure of the Heisenberg group} 
The Heisenberg group $\mathbbm{H}^N$ is a homogeneous (nilpotent)  Lie group, which is defined as the triplet $\big(\mathbb{R}^{2N+1},\circ,\{\Phi_\lambda\}\big)$, where $\mathbb{R}^{2N+1}$ is the underlying space, $\circ$ denotes the group law, and $\{\Phi_\lambda\}$ denotes the family of dilations. A Point $\xi \in \mathbbm{H}^N$ is denoted by
\begin{align}
	\xi\coloneqq(x, y,w)=(x_1,\dots,x_N,y_1,\dots,y_N,w),	
\end{align}
where $x_i,\,y_i\in\mathbb{R}$ for $1\leq i\leq N,$ and $w\in\mathbb{R}.$
The group law $\circ$ is defined for $\xi=(x,y,w)$ and $\xi'=(x',y',w')$ as
\begin{align}
	\xi\circ \xi'&=\big(x+x',y+y',w+w'+2\langle y,x'\rangle- 2\langle x,y'\rangle\big),
\end{align}
or in coordinates,	\begin{align}
	\xi\circ\xi'=\bigg(x_1+x_1',\dots,x_N+x_N',y_1+y_1',\dots,y_N+y_N',w+w'+2\sum_{i=1}^{N}\big(y_ix_i'-x_iy_i'\big)\bigg).
\end{align}	
The identity element is $\textbf{0}=(0,0,0),$ and the inverse $\xi=(x,y,w)$ is  $\xi^{-1}=(-x,-y,-w).$ Also, the dilation group
\begin{align}\label{dil}
	\Phi_\lambda: \mathbb{R}^{2N+1}\longrightarrow\mathbb{R}^{2N+1},
\end{align}
is given by
\begin{align}
	\Phi_\lambda(\xi)\coloneqq\big(\lambda x,\lambda y,\lambda^2 w\big).
\end{align} 
These operations endow $\mathbbm{H}^N$ with a non-commutative and anisotropic dilation structure, central to our analysis.

\subsection{Vector fields, norm and the sub-Laplacian}
The Heisenberg group $\mathbbm{H}^N,$ also known as the Heisenberg-Weyl group is a step-2 nilpotent Lie group with Lie algebra generated by the vector fields:
\begin{align}
	X_i=\partial_{x_i}+2y_{i}\partial_w,\, Y_{i}=\partial _{y_i}-2x_i\partial_w, \, 1\leq i\leq N,\, W=\partial_w.
\end{align}
Given a domain $\Omega\subset \mathbbm{H}^N,$ and a function $u\in C^1(\Omega,\mathbb
{R}),$ the horizontal gradient (or Heisenberg gradient) is defined as
\begin{align}
	\nabla_{H} u(\xi)\coloneqq\big(X_1u(\xi),\dots,X_{N}u(\xi),Y_1u(\xi),\dots,Y_Nu(\xi)\big).
\end{align}
A key structural property of $\mathbbm{H}^N$ lies in the commutation relations between these vector fields. For $1\leq i\leq N,$ we have 
\begin{align}
	[X_i,Y_{i}]&=X_iY_{i}-Y_{i}X_i\\
	&=(\partial_{x_i}+2y_{i}\partial_w)(\partial_{y_{i}}-2x_i\partial_w)-(\partial_{y_{i}}-2x_i\partial_w)(\partial_{x_i}+2y_{i}\partial_w)\\
	&=-{4}\partial_w, \text{ for }1\leq i\leq N.
\end{align}
This implies that the Lie algebra generated by $\{X_i,Y_i\}_{{i=1}}^N$ and their commutators spans the full tangent space:
\begin{align}
	\text{rank}\big(\text{Lie}\{X_1,X_2,\dots,X_{N}, Y_1,Y_2,\dots,Y_{N}\}(0,0)\big)=2N+1,
\end{align}
which satisfies the H\"ormander's rank condition. The family of dilations $\{\Phi_\lambda\}_{\lambda>0}$ introduced earlier has Jacobian determinant $\lambda^Q,$ where \begin{align}
	Q:=2N+2
\end{align} 
is called the homogeneous dimension of $\mathbbm{H}^N.$

Throughout this work, we adopt the standard \textit{homogeneous norm} (also known as the Kor\'anyi norm) on $\mathbbm{H}^N$$\colon$ 	
\begin{align}\label{norm}
	\|\xi\|_{H}\text{:=}\,\bigg[\bigg(\sum_{i=1}^N \big(x_i^2+y_i^2\big)^2\bigg)+w^2\bigg]^{\frac{1}{4}},
\end{align} 
where 	
\begin{align}
	\xi =(x_1,\dots,x_N,y_1,\dots,y_N,w)\in\mathbbm{H}^N.
\end{align} 
This is a symmetric homogeneous norm on $\mathbbm{H}^N.$ Given $\xi_0\in \mathbbm{H}^N\text{ and radius }\, R>0,$ we define the Kor\'anyi-ball (also called Heisenberg ball) centered at $\xi_0$ by
\begin{align}\label{Ball}
	B_R^{H}(\xi_0)\text{:=}\,\bigg\{\xi\in \mathbbm{H}^N\,:\,\|\xi_0^{-1}\circ\xi\|_{H}<R\bigg\}.
\end{align}

The sub-Laplacian (also called the Heisenberg Laplacian or Kohn-Laplace operator) on $\mathbbm{H}^N$ is the second-order self-adjoint differential operator defined by
\begin{align}
	\Delta_{H}\coloneqq\sum_{i=1}^NX_i^2+Y_i^2.	
\end{align}
Expressed in terms of Euclidean derivatives, this becomes
\begin{align}
	\Delta_{H}=\sum_{i=1}^N\frac{\partial^2}{\partial x_i^2}+\frac{\partial^2}{\partial y_i^2}+4y_i\frac{\partial^2}{\partial x_i\partial w}-4x_i\frac{\partial^2}{\partial y_i\partial w}+4\big(x_i^2+y_i^2\big)\frac{\partial^2}{\partial w^2}.
\end{align} 
\subsection{Heat semigroup}	
The heat semigroup associated with $-\Delta_{H}$ is given by
\begin{align}\label{hs}
	e^{t\Delta_{H}}f(\xi)&\coloneqq f*h(\cdot,t)(\xi)\\
	&=\int_{\mathbbm{H}^N}h(\eta^{-1}\circ\xi,t)f(\eta)d\eta,
\end{align}
for $f\in L^p(\mathbbm{H}^N),$ see \cite{Gre}. The function $h$ is the corresponding integral kernel, satisfying  the following two-sided estimate:
\begin{align}\label{Esti}
	ct^{-\frac{Q}{2}}\exp\left(-\frac{C\|\xi\|^2_{H}}{t}\right) \leq h(\xi,t)\leq Ct^{-\frac{Q}{2}}\exp\left(-\frac{c\|\xi\|^2_{H}}{t}\right),
\end{align}
for some constants $c,C>0$ and all $t>0,$ $\xi\in\mathbbm{H}^N.$ One may also see \cite{Varo} for the details. Also, we have the following result:
\begin{lem}\cite{Gre}
	The heat semigroup associated with $-\Delta_{H}$ satisfies the following estimate:
	\begin{align}\label{infty}
		\|e^{t\Delta_{H}}g\|_{L^\infty(\mathbbm{H}^N)}\leq \|g\|_{L^\infty(\mathbbm{H}^N)},
	\end{align}
	for any $g\in L^\infty(\mathbbm{H}^N)$ and $t>0.$
\end{lem}	

\subsection{Auxiliary tools}
We recall some standard auxiliary results used in later sections.
\begin{defn}[Beta function]
	The Beta function $B(z_1,z_2)$ is a function of two complex variables $z_1,\,z_2,$ defined as
	\begin{align}\label{Beta}
		B(z_1,z_2)=\int_0^1t^{z_1-1}(1-t)^{z_2-1}dt,\,\, \mathcal{R}e(z_i)>0,\,i=1,2.
	\end{align}
	This integral is also known as the Eulerian integral of the first kind.
\end{defn}
Also, let us recall the definition of left and right Riemann–Liouville fractional integrals occurring in our model equation and further analysis.
\begin{defn}[Riemann-Liouville fractional integral \cite{Kilbas}] For a given integrable function $u(t),$ $t\in (0,T),$ the left and right Riemann–Liouville fractional integrals of order $\gamma\in (0,1)$ are defined as 
	\begin{align}
		I_{0+}^\gamma u(t)=\frac{1}{\Gamma(\gamma)}\int_0^t(t-\tau)^{\gamma-1}u(\tau)d\tau,
	\end{align}
	and
	\begin{align}
		I_{T-}^\gamma u(t)=\frac{1}{\Gamma(\gamma)}\int_t^T(\tau-t)^{\gamma-1}u(\tau)d\tau,
	\end{align}
	respectively. Moreover, $I_{T-}^\gamma u, I_{0+}^\gamma u\To u$ a.e. as $\gamma\To 0,$ see \cite{Kilbas}. In view of it, we use the convention that $I_{T-}^0 u(t), I_{0+}^0 u(t)=u(t), t\in (0,T).$ 
\end{defn}


\begin{thm}[Lemma 2.2 \cite{Dou}]\label{Hardy} Let $f\in C_c^\infty(\mathbbm{H}^N)\setminus\{0\}.$ Then
	\begin{align}
		\int_{\mathbbm{H}^N}\frac{(f(\xi))^2}{\|\xi\|_H^{2}}\psi(\xi)d\xi\leq \left(\frac{2}{Q-2}\right)^2\int_{\mathbbm{H}^N}|\nabla_{H}f(\xi)|^2d\xi,
	\end{align}
	where $\psi=|\nabla_{H}\|\cdot\|_{H}|^2.$ Moreover, the constant $\left(\frac{2}{Q-2}\right)^2$ appearing in the inequality is sharp.
\end{thm}

Below, we state the subelliptic variant of the Sobolev inequality given by Folland and Stein \cite{FollSte}.
\begin{thm}[Folland-Stein-Sobolev inequality \cite{FollSte}]\label{Thm Folst}
	Let $1<p<\infty.$ Then, there exists a constant $C=C(p,Q)$ such that 
	\begin{align}
		\left(\int_{\mathbbm{H}^N}|\phi|^{\frac{pQ}{Q-p}}d\xi\right)^{\frac{Q-p}{pQ}} \leq C(p,Q)\left(\int_{\mathbbm{H}^N}|\nabla_H\phi|^p\right)^{\frac{1}{p}}
	\end{align}	
	for all $\phi\in C_c^\infty(\mathbbm{H}^N).$
\end{thm}

\subsection{Weighted transformations}\label{WT sub}
In view of the above Hardy inequality, we define 
\begin{align}\label{lambdamu}
	\mu(\lambda)\coloneqq\frac{Q-2}{2}-\sqrt{\left(\frac{Q-2}{2}\right)^2-\lambda},
\end{align}
for $\lambda\in \big(0,\big(\frac{Q-2}{2}\big)^2\big),$ where $Q=2N+2$ is the homogeneous dimension of $\mathbb{H}^N.$ It is easy to check that
\begin{align}
	\mu^2(\lambda)-(Q-2)\mu(\lambda)+\lambda=0.
\end{align}
In view of the above identity, we compute
\begin{align}
	\nabla_{H}(\|\cdot\|_{H}^{-\mu(\lambda)})&=-\mu(\lambda)(\|\cdot\|^{-\mu(\lambda)-1}_{H})\nabla_{H}\|\cdot\|_{H}. 
\end{align}
Also,
\begin{align}
	\Delta_{H}\big(\|\cdot\|_{H}^{-\mu(\lambda)}\big)&=-\mu(\lambda)(-\mu(\lambda)+Q-2)\big(\|\cdot\|^{-\mu(\lambda)-2}_{H}\big)\big|\nabla_{H}\|\cdot\|_{H}\big|^2,
\end{align}
equivalently,
\begin{align}\label{N1}
	\|\cdot\|_{H}^{\mu(\lambda)}\Delta_{H}(\|\cdot\|_{H}^{-\mu(\lambda)})&=-\mu(\lambda)(-\mu(\lambda)+Q-2)\big(\|\cdot\|^{-2}_{H}\big)\big|\nabla_{H}\|\cdot\|_{H}\big|^2\\
	&=-\lambda(\|\cdot\|^{-2}_{H})\psi,
\end{align}
where 
\begin{align}\label{psi}
	\psi(\cdot):=|\nabla_{H}\|\cdot\|_{H}|^2.
\end{align}
\noindent Next, let us consider the transformation \begin{align}
	v(\xi,t)=\|\xi\|_{H}^{{\mu(\lambda)}}u(\xi,t).
\end{align} 
This infers
\begin{align}\label{L^s}
	L_{H}v&\coloneqq-\nabla_{H}(\|\cdot\|^{-2\mu(\lambda)}_{H}\nabla_{H}v)\\
	&=-\|\cdot\|^{-2\mu(\lambda)}_{H}\Delta_{{H}}v-(\nabla_{H}v)\big(\nabla_{H}\|\cdot\|^{-2\mu(\lambda)}_{H}\big)\\
	&=-\|\cdot\|^{-2\mu(\lambda)}_{H}\Delta_{{H}}v-2\|\cdot\|^{-\mu(\lambda)}_{H}(\nabla_{H}v)\big(\nabla_{H}\|\cdot\|^{-\mu(\lambda)}_{H}\big),
\end{align}
where we have used the identity
\begin{align}
	\nabla_{H}\|\cdot\|^{-2\mu(\lambda)}_{H}=2\|\cdot\|^{-\mu(\lambda)}\big(\nabla_{H}\|\cdot\|^{-\mu(\lambda)}_{H}\big).
\end{align}
Furthermore, utilizing \eqref{N1}, we get
\begin{align}
	L_{H}v&=-\|\cdot\|^{-2\mu(\lambda)}_{H}\Delta_{{H}}v-2\|\cdot\|^{-\mu(\lambda)}_{H}(\nabla_{H}v)\big(\nabla_{H}\|\cdot\|^{-\mu(\lambda)}_{H}\big)\\
	&=-\|\cdot\|^{-2\mu(\lambda)}_{H}\Delta_{{H}}v-2\|\cdot\|^{-\mu(\lambda)}_{H}(\nabla_{H}v)\big(\nabla_{H}\|\cdot\|^{-\mu(\lambda)}_{H}\big)\\
	&\qquad-\|\cdot\|_{H}^{-\mu(\lambda)}\Delta_{H}\big(\|\cdot\|_{H}^{-\mu(\lambda)}\big)v-\lambda\big(\|\cdot\|^{-2-2\mu(\lambda)}_{H}\big)\psi v\\
	&=-\|\cdot\|^{-\mu(\lambda)}_{H}\left(\|\cdot\|^{-\mu(\lambda)}_{H}\Delta_{{H}}v+2(\nabla_{H}v)\big(\nabla_{H}\|\cdot\|^{-\mu(\lambda)}_{H}\big)+v\Delta_{H}\big(\|\cdot\|_{H}^{-\mu(\lambda)}\big)\right)\\
	&\qquad-\lambda\big(\|\cdot\|^{-2-2\mu(\lambda)}_{H}\big)\psi v\\
	&=-\|\cdot\|^{-\mu(\lambda)}_{H}\big(\Delta_{{H}}\big(v\|\cdot\|^{-\mu(\lambda)}_{H}\big)\big)-\lambda\big(\|\cdot\|^{-2-2\mu(\lambda)}_{H}\big)\psi v\\
	&=-\|\cdot\|^{-\mu(\lambda)}_{H}\Delta_{{H}}u-\lambda\big(\|\cdot\|^{-2-2\mu(\lambda)}_{H}\big)\psi \|\cdot\|_H^{\mu(\lambda)}u\\
	&=-\|\cdot\|^{-\mu(\lambda)}_{H}\left(\Delta_{{H}}u+\lambda\big(\|\cdot\|^{-2}_{H}\big)\psi u\right).
\end{align}
This yields 
\begin{align}
	L_{H}v=-\|\cdot\|^{-\mu(\lambda)}_{H}\left(\Delta_{{H}}u+\lambda \frac{\psi u}{\|\cdot\|^2_{H}}\right),
\end{align}
where $\psi$ and $L_{H}$ are given by \eqref{psi} and \eqref{L^s}, respectively. 
This together with \eqref{eq 0.1} infers that $v$ solves the following problem:
\begin{align}\label{eq equi}
	\begin{cases}
		\|\cdot\|_{H}^{-2\mu(\lambda)}\partial_tv+L_{H}v=\frac{1}{\Gamma(\gamma)}\|\cdot\|_{H}^{-(p+1)\mu(\lambda)}\int_0^t(t-s)^{\gamma-1}|v(s)|^{p}ds+\|\cdot\|_{H}^{-\mu(\lambda)}t^\alpha f &\text{in } \,\mathbbm{H}^N\times (0,T),\\
		v(0,\cdot)=\|\cdot\|_{H}^{{\mu(\lambda)}}u_0(\cdot) &\text{in } \,\mathbbm{H}^N,
	\end{cases}	
\end{align}
where $\mu(\lambda)$ is defined in \eqref{lambdamu}. 
%
%

\subsection{Weak solution}
In view of the above transformation, one can say that \eqref{eq 0.1} is equivalent to \eqref{eq equi}. Using this fact, we define the notion of weak solutions to \eqref{eq 0.1} as follows:
\begin{defn}[Weak solution]\label{Weak} Let $v(0)\in L^1_{\text{loc}}(\mathbbm{H}^N, \|\xi\|^{-2\mu(\lambda)}d\xi),f\in L^1_{\text{loc}}(\mathbbm{H}^N, \|\xi\|^{-\mu(\lambda)}d\xi).$ Then $v\in L^p_{\text{loc}}((0,T),L^p_{\text{loc}}(\mathbbm{H}^N,\|\xi\|_H^{-{(p+1)\mu(\lambda)}}d\xi))\cap L^1_{\text{loc}}((0,T),L^1_{\text{loc}}(\mathbbm{H}^N,\|\xi\|_H^{-2\mu(\lambda)}d\xi)) $ is called a local-in-time weak solution to \eqref{eq equi} (and equivalently, to \eqref{eq 0.1}) if 
	\begin{align}
		-\int_0^T\int_{\mathbbm{H}^N}	v(\|\xi\|_{H}^{-{2\mu(\lambda)}}\partial_t\phi-L_{H}\phi)d\xi dt &=\frac{1}{\Gamma(\gamma)}\int_0^T\int_{\mathbbm{H}^N}\|\xi\|_{H}^{{-(p+1)\mu(\lambda)}}|v|^pI^\gamma_{T^-}\phi d\xi dt\\
		&\qquad +\int_{\mathbbm{H}^N}\|\xi\|_{H}^{{-2\mu(\lambda)}}v(0,\xi)\phi(0,\xi)d\xi\\
		&\qquad+\int_0^T\int_{\mathbbm{H}^N}\|\xi\|_{H}^{{-\mu(\lambda)}}t^\alpha f\phi d\xi dt
	\end{align}
	holds for any non-negative test function $\phi \in C_c^1([0,T), C_c^2(\mathbbm{H}^N).$ Moreover, $v$ is called global-in-time weak solution if $T=\infty.$
	
\end{defn}
\section{Local existence and instantaneous blow-up}\label{Local wel}
In this section, we investigate the local well-posedness and nonexistence of solutions to problem \eqref{eq 0.1} on the Heisenberg group. Our analysis reveals a sharp contrast between the problem with and without the Hardy potential. When the Hardy potential is present, and $\gamma=0,$ the nonlinear term can induce instantaneous blow-up, preventing the existence of local-in-time weak solutions once the exponent exceeds a critical threshold. Moreover, under a mild temporal monotonicity assumption on the solution, we extend this nonexistence result to the case involving a nonlocal memory term, that is, $\gamma>0.$

On the other hand, in the absence of the Hardy potential, we establish local-in-time existence and uniqueness of mild solutions in $L^\infty(\mathbbm{H}^N)$ by means of a fixed-point argument.

We begin with the case $\gamma=0$, where the Hardy potential induces instantaneous blow-up for sufficiently strong nonlinearities.	
\begin{thm}\label{Thm Blowup O4}
	Let $\gamma=0,$  $\alpha\in (-1,\infty),$ and let  $0\leq u_0,f\in L^\infty(\mathbbm{H}^N).$ Then, for any $p>
	\frac{\mu(\lambda)+2}{\mu(\lambda)},$ problem \eqref{eq 0.1} does not admit a local-in-time solution.
\end{thm}

To handle the case when $\gamma>0$, we introduce a mild temporal monotonicity condition ensuring that the solution does not decay too rapidly in short time intervals.
\begin{enumerate}[label=\upshape {(H\arabic*)}, ref=(H\arabic*)]	
	\item\label{H1} Assume that there exist constants $\delta>0,$ $\theta\in (0,1]$ such that for each $t\in (0,T),$ and $s\in [t-\delta,t],$ we have
	\begin{align}
		u(s,\xi)\geq \theta u(t,\xi) \text{ for all }\xi\in\mathbbm{H}^N.
	\end{align} 
\end{enumerate}	
Under the above assumption, we obtain a corresponding nonexistence result for problem \eqref{eq 0.1}.
\begin{thm}\label{Thm Blowup O5}
	Let $\gamma>0,$ $\alpha\in (-1,\infty),$ and let  $0\leq u_0,f\in L^\infty(\mathbbm{H}^N).$ Assume that \ref{H1} holds. Then, for any $p>
	\frac{\mu(\lambda)+2}{\mu(\lambda)},$ problem \eqref{eq 0.1} does not admit a local-in-time weak solution.
\end{thm}
\begin{rem}
	Condition \ref{H1} is a mild temporal monotonicity assumption over short intervals. It ensures that the memory term remains comparable to its instantaneous value.
	It generalizes the standard monotonicity condition $u_t\ge 0$ over some interval (as used by Li et al. \cite{Li}), corresponding to the case $\theta=1$. Hence, \ref{H1} can be viewed as a variant of this classical assumption.
\end{rem}	
We now turn to the case without the Hardy potential, where local solutions can be constructed by fixed-point arguments. Consider
\begin{align}\label{eq without}
	\begin{cases}
		\partial_tu-\Delta_{H}u=\frac{1}{\Gamma(\gamma)}\int_0^t(t-\tau)^{\gamma-1}|u(\tau)|^{p}d\tau+t^\alpha f &\text{ in } \,\mathbbm{H}^N\times (0,T)\\
		u_0(\xi)=u(0,\xi) &\text{ in }\mathbbm{H}^N.
	\end{cases}
\end{align}
Now, we give the notions of mild and weak solutions to the above problem.
\begin{defn}[Mild solution]\label{Mild}
	A function $u\in C([0,T],L^\infty(\mathbbm{H}^N))$ is called a local-in-time mild solution to \eqref{eq without} if
	\begin{align}\label{eq Mild}
		u(t)=e^{t\Delta_{H}}u_0+\int_0^te^{(t-\tau)\Delta_{H}}\left(\frac{1}{\Gamma (\gamma)}\int_0^\tau(\tau-s)^{-\gamma}|u(s)|^{p}ds+\tau^\alpha f\right)d\tau
	\end{align}
	holds for any $t\in (0,T].$
\end{defn}

\begin{defn}[Weak solution]\label{Weak2}
	Let $u_0,f\in L^1_{\text{loc}}(\mathbbm{H}^N).$ Then $u\in L^p_{\text{loc}}((0,T),L^p_{\text{loc}}(\mathbbm{H}^N))$ is called a local-in-time weak solution to \eqref{eq without} if 
	\begin{align}
		-\int_0^T\int_{\mathbbm{H}^N}u(\partial_t\phi+\Delta_{H}\phi)d\xi dt =\int_0^T\int_{\mathbbm{H}^N}|u|^pI_{T^-}^\gamma\phi d\xi dt +\int_{\mathbbm{H}^N}u_0\phi(0,\cdot)d\xi +\int_0^T\int_{\mathbbm{H}^N} \tau^\alpha f\phi d\xi d\tau,
	\end{align}
	holds for any non-negative test function $\phi\in C_c^1([0,T),C_c^2(\mathbbm{H}^N)).$ 
	Moreover, $u$ is called a global-in-time weak solution if $T=+\infty.$ 
\end{defn}
The following lemma links these two notions of solutions to \eqref{eq without}. 
\begin{lem}\label{mild to weak}
	Let $u_0\in L^\infty(\mathbbm{H}^N).$ Let $u$ be a mild solution to \eqref{eq without} in $\mathbbm{H}^N\times [0,T),$ for $T>0.$ Then $u$ is also a weak solution to \eqref{eq without} in $\mathbbm{H}^N\times [0,T).$
\end{lem}
\noindent We avoid the proof as the arguments follow similar lines of proof as in Lemma 3.2 \cite{Berik}.
Now, we state the local-well-posedness result, which reads as follows:
\begin{thm}\label{Local}
	Let $\gamma>0,$ $\alpha\in(-1,\infty)\setminus\{0\},$ $p>1$ and $u_0,f\in L^\infty(\mathbbm{H}^N),$ with $u_0\geq 0.$ Then, the following holds:
	\begin{enumerate}
		\item [(i)] There exists a fixed $T>0,$ and a unique mild solution $u\in C([0,T],L^\infty(\mathbbm{H}^N))$ to \eqref{eq without}. 
		\item [(ii)] The solution can be uniquely extended to a maximum interval $[0,T_{\text{max}}).$ Moreover, if $T_{\text{max}}$ is finite, then $\|u(t)\|_{L^\infty(\mathbbm{H}^N)}\To\infty$ as $t\To T_{\text{max}}.$
		\item [(iii)] Let $u_0,$ $f\in L^\infty(\mathbbm{H}^N)\cap L^q(\mathbbm{H}^N)$ for $q\in [1,\infty].$ Then
		\begin{align}
			u\in C([0,T_{\text{max}}),L^\infty(\mathbbm{H}^N))\cap C([0,T_{\text{max}}),L^q(\mathbbm{H}^N)).
		\end{align}
	\end{enumerate} 
\end{thm}	

\begin{rem}
	When $\gamma=0,$ problem \eqref{eq without} is well-studied by the second and third-named authors in \cite{Oza Sur 2}. In fact, in this regime, we have the existence of a global-in-time solution to \eqref{eq without}. 
\end{rem}	
We now give the proof of Theorem \ref{Thm Blowup O4}. We follow the arguments used in \cite{Abd, Opt, Talwar2}.\\

\noindent \textbf{Proof of Theorem \ref{Thm Blowup O4}.}
Let, if possible, there exists a local-in-time solution to \eqref{eq 0.1}.
Here, we follow the arguments of Theorem 6.3 \cite{Opt} which were further adapted for stratified Lie groups in \cite{Talwar2}. We make use of a comparison principle together with the Harnack inequality (Theorem 1.1 \cite{Avelin}). Let us define an iterated system 
\begin{align}\label{eq itr}
	\begin{cases}
		\partial_tu_n-\Delta_{H}u_n=\lambda \frac{\psi u_{n-1}}{\big(\|\cdot\|_H^2+\frac{1}{n}\big)}+|u_{n-1}|^p+t^\alpha f&\text{in } \,B^H_r\times (0,T),\\
		u_n(\xi,0)=u_0(\cdot) &\text{in } \,B^H_r,\\
		u_n(\xi,t)=0 &\text{on } \,\partial B^H_r\times (0,T),
	\end{cases}	
\end{align}
where $r\in (0,1)$ is chosen such that $u_n>1$ in $B^H_r\times [0,T]$ and $\{u_n\}$ is an increasing sequence of solutions. The existence of such a sequence is assured by the comparison principle (Theorem 2.1 \cite{Rock}) and the fact that for each problem, the R.H.S. is a fixed function.

\noindent Now, let us consider a compactly supported function $\phi\in C_c^\infty (B^H_r)$ and multiply $\frac{|\phi|^2}{u_n}$ to the equation \eqref{eq itr} followed by integrating over $[0,T]\times B^H_r,$ we get
\begin{align}
	\int_{0}^{T}\int_{B^H_r}\frac{|\phi|^2}{u_n}\left(\partial_tu_n-\Delta_{H}u_n\right)d\xi ds=\int_{0}^{T}\int_{B^H_r}\frac{|\phi|^2}{u_n}\left(\lambda \frac{\psi u_{n-1}}{\big(\|\cdot\|_H^2+\frac{1}{n}\big)}+|u_{n-1}|^p+t^\alpha f\right)d\xi ds.	
\end{align} 
This infers
\begin{align}\label{W1}
	\int_{0}^{T}\int_{B^H_r}\frac{|\phi|^2}{u_n}\left(\partial_tu_n-\Delta_{H}u_n\right)d\xi ds\geq\int_{0}^{T}\int_{B^H_r}|u_{n-1}|^{p-1}\phi^2d\xi dt.
\end{align}
Simple computations provide the following expressions for any sufficiently regular positive function $g,$ we have
\begin{align}
	\frac{\partial_tg}{g}=\partial_t(\log g),
\end{align}
and
\begin{align}
	\nabla_{H}\cdot\left(\frac{\nabla_{H}g}{g}\right)=\frac{\Delta_{H}g}{g}-\frac{|\nabla_{H}g|^2}{g^2}.
\end{align}
These further infer 
\begin{align}
	\frac{\phi^2}{g}\left(\partial_tg-\Delta_{H}g\right)=\phi^2\partial_t(\log g)-\phi^2\nabla_{H}\cdot\left(\frac{\nabla_{H}g}{g}\right)-\phi^2\frac{|\nabla_{H}g|^2}{g^2}.
\end{align}
Using the above relation in \eqref{W1} gives
\begin{align}\label{la1}
	\qquad\int_{B_{H}(r)}{\phi^2}\log(u_n(\xi,T))d\xi+2\int_0^T\int_{B^H_r}\phi\nabla_{H}\phi\cdot\frac{\nabla_{H}u_n}{u_n}d\xi dt&-\int_0^T\int_{B^H_r}|\phi|^2\frac{|\nabla_{H}\phi|^2}{u_n^2}d\xi dt\\
	&\quad\geq \int_0^T\int_{B^H_r}|u_{n-1}|^{p-1}|u_n|^2d\xi dt,
\end{align} 
for any fixed time $T.$
Now, as an application of Picone identity ( \cite[Lemma 2.1]{Dou}) with $p=2,$ we have that
\begin{align}
	|\nabla_H \phi|^2\geq 2\frac{\phi}{u_n}\nabla_H\phi\cdot\nabla_Hu_n-\frac{\phi^2}{u_n^2}|\nabla_H\phi|^2.
\end{align}
Utilizing this in \eqref{la1} provides
\begin{align}\label{W22}
	\int_{B^H_r}{\phi^2}\log(u_n(\xi,T))d\xi+\int_{0}^{T}\int_{B^H_r}|\nabla_{H}\phi|^2d\xi dt\geq \int_{0}^{T}\int_{B^H_r}|u_{n-1}|^{p-1}\phi^2 d\xi dt.
\end{align}
Next, taking the limit $n\To\infty,$ we get
\begin{align}\label{W2}
	\int_{B^H_r}{\phi^2}\log(u(\xi,T))d\xi+\int_{0}^{T}\int_{B^H_r}|\nabla_{H}\phi|^2d\xi dt\geq \int_{0}^{T}\int_{B^H_r}|u|^{p-1}\phi^2 d\xi dt.
\end{align}
Furthermore, an application of the comparison principle ensures the non-negativity of $v(\xi,t)=\|\xi\|_H^{{\mu(\lambda)}}u(\xi,t),$ whereas an application of the Harnack inequality ( \cite[Theorem 1.1]{Avelin}) ensures that
\begin{align}
	u(\xi,t)\geq C(r)\|\xi\|_H^{-{\mu(\lambda)}}.
\end{align}
Using this estimate in \eqref{W2} yields
\begin{align}\label{W3}
	\int_{B^H_r}{|\phi|^2}\log(u(\xi,T))d\xi+T\int_{B^H_r}|\nabla_{H}\phi|^2d\xi\geq C(r)\int_{0}^{T}\int_{B^H_r}\|\xi\|_{H}^{-\mu(\lambda)(p-1)}\phi^2d\xi dt.
\end{align}
Also, by an application of Holder's inequality followed by the subelliptic variant of the Sobolev-type inequality (Theorem \ref{Thm Folst}) with $p=2,$ we have that
\begin{align}
	\int_{B^H_r}{|\phi|^2}\log(u(\xi,T))d\xi&\leq \left(\int_{B^H_r}\left(\log(u(\xi,T))\right)^{\frac{Q}{2}}d\xi\right)^\frac{2}{Q}\left(\int_{B^H_r}{|\phi|^{\frac{2Q}{Q-2}}}d\xi\right)^{\frac{Q-2}{Q}}\\
	&\leq C\left(\int_{B^H_r}\left(\log(u(\xi,T))\right)^{\frac{Q}{2}}d\xi\right)^\frac{2}{Q}\left(\int_{B^H_r}{|\nabla_{H}\phi|^{2}}d\xi\right).
\end{align}
Finally, using the above estimate in \eqref{W3} gives
\begin{align}
	C\left(\int_{B^H_r}\left(\log(u(\xi,T))\right)^{\frac{Q}{2}}d\xi\right)^\frac{2}{Q}\int_{B^H_r}{|\nabla_{H}\phi|^{2}}d\xi&+T\int_{B^H_r}|\nabla_{H}\phi|^2d\xi\\
	&\geq \int_{0}^{T}\int_{B^H_r}\|\xi\|_{H}^{-\mu(\lambda)(p-1)}\phi^2d\xi dt.
\end{align}
Equivalently,
\begin{align}
	\left(T+C\left(\int_{B^H_r}\left(\log(u(\xi,T))\right)^{\frac{Q}{2}}d\xi\right)^\frac{2}{Q}\right)\int_{B^H_r}{|\nabla_{H}\phi(\xi)|^{2}}d\xi\geq C\int_{0}^{T}\int_{B^H_r}\|\xi\|_{H}^{-\mu(\lambda)(p-1)}\phi^2(\xi)d\xi dt.
\end{align}
However, the above inequality cannot hold when $(p-1)\mu(\lambda)>2,$ since in this case the weight $\|\xi\|_{H}^{-\mu(\lambda)(p-1)}\phi^2(\xi)$ is more singular than the Hardy-critical weight $\|\xi\|_{H}^{-2}\phi^2(\xi),$ thereby contradicting the sharp Hardy inequality. Hence, the claim follows.\qed\\

The proof of Theorem \ref{Thm Blowup O5} is as follows:\\

\noindent \textbf{Proof of Theorem \ref{Thm Blowup O5}.}
Mimicking the similar lines of proof as Theorem \ref{Thm Blowup O4} till the inequality \eqref{W2} followed by the assumption \ref{H1} infers
\begin{align}
	\int_{B^H_r}{|\phi|^2}\log(u(\xi,T))d\xi+\int_{0}^{T}\int_{B^H_r}|\nabla_{H}\phi|^2d\xi&\geq \int_{0}^{T}\int_{B^H_r}\left(\frac{1}{\Gamma(\gamma)}\int_0^t(t-\tau)^{\gamma-1}|u(\tau)|^{p}d\tau\right)\frac{\phi^2}{u} d\xi dt\\
	&\geq \int_{0}^{T}\int_{B^H_r}\left(\frac{1}{\Gamma(\gamma)}\int_{t-\delta}^{t}(t-\tau)^{\gamma-1}|\theta u(t)|^{p}d\tau\right)\frac{\phi^2}{u}d\xi dt\\
	&=\frac{\theta^p\delta^\gamma}{\gamma\Gamma(\gamma)}\int_{0}^{T}\int_{B^H_r}|u(t)|^{p-1}{\phi^2}d\xi dt\\
	&= \frac{\theta^p\delta^\gamma}{\Gamma(\gamma+1)}\int_{0}^{T}\int_{B^H_r}|u(t)|^{p-1}\phi^2d\xi dt.
\end{align}
Again, by an application of the comparison principle ( \cite[Theorem 2.1]{Rock}) together with the Harnack inequality ( \cite[Theorem 1.1]{Avelin}), we get 
\begin{align}
	\int_{B^H_r}{|\phi|^2}\log(u(\xi,T))d\xi+T\int_{B^H_r}|\nabla_{H}\phi|^2 d\xi
	&\geq \frac{\theta^p\delta^\gamma}{\Gamma(\gamma+1)}\int_{0}^{T}\int_{B^H_r}|u(t)|^{p-1}\phi^2d\xi dt\\
	&\geq C(r)\frac{\theta^p\delta^\gamma}{\Gamma(\gamma+1)}\int_{0}^{T}\int_{B^H_r}\|\xi\|_{H}^{-\mu(\lambda)(p-1)}\phi^2 d\xi dt. 
\end{align}
This gives a contradiction to the Hardy inequality as in the previous theorem. This completes the proof.
\qed\\

Finally, we give a proof of the local-wellposedness result.	\\

\noindent \textbf{Proof of Theorem \ref{Local}}
(i) Firstly, for \begin{align}
	\delta(u_0,f)\coloneqq\max\{\|u_0\|_{L^\infty(\mathbbm{H}^N)}, \|f\|_{L^\infty(\mathbbm{H}^N)}\},
\end{align} 
let us consider the following Banach space
\begin{align}\label{e13}
	\Theta^s_T\coloneqq\bigg\{v\in C([0,T],L^\infty(\mathbbm{H}^N)):\|v\|_{L^\infty(\mathbbm{H}^N)}\leq 2\delta(u_0,f), v(0)=u_0\bigg\},
\end{align}
where 
\begin{align}
	\|v\|_{\Theta^s_T}=\|v\|_{L^\infty((0,T),L^\infty(\mathbbm{H}^N))}.
\end{align} 
For any $v\in \Theta_T^s,$ we define a map
\begin{align}\label{Phii}
	\Phi(v)(t)\coloneqq e^{t\Delta_{H}}v(0)+\int_0^te^{(t-\tau)\Delta_{H}}\left(\frac{1}{\Gamma(\gamma)}\int_0^\tau(\tau-\tau_1)^{\gamma-1}|v(\tau_1)|^{p}d\tau_1+\tau^\alpha f\right)d\tau.
\end{align}
We first prove that for any $v\in\Theta_T^s,$ we have $\Phi(v)\in\Theta_T^s.$ In light of \eqref{infty}, we have that
\begin{align}
	\|\Phi(v)\|_{\Theta_T^s}&\leq \|e^{t\Delta_{H}}v(0)\|_{L^\infty(\mathbbm{H}^N)}+\Bigg\|\int_0^te^{(t-\tau)\Delta_{H}}\left(\frac{1}{\Gamma(\gamma)}\int_0^\tau(\tau-\tau_1)^{\gamma-1}|v(\tau_1)|^{p}d\tau_1+\tau^\alpha f\right)d\tau\Bigg\|_{L^\infty(H^N)}\\
	&\leq \|v(0)\|_{L^\infty(\mathbbm{H}^N)}+C(\gamma)\int_0^t(t-\tau)^{\gamma}\|v\|^p_{L^\infty((0,t),L^\infty(\mathbbm{H}^N))}d\tau+\frac{t^{\alpha+1}}{\alpha+1}\|f\|_{L^\infty(\mathbbm{H}^N)}\\
	&\leq  \bigg(1+\frac{t^{\alpha+1}}{\alpha+1}\bigg)\delta(u_0,f)+C(\gamma)2^p\delta(u_0,f)^pt^{1+\gamma}.
\end{align}
Observe that for $t=T>0,$ sufficiently small, we get
\begin{align}
	\|\Phi(v)\|_{\Theta_T^s}\leq 2\delta(u_0,f).
\end{align}
This yields $\Phi(v)\in \Theta_T^s.$
We claim that for sufficiently small $T>0,$ $\Phi$ is a contraction map from $\Theta^s_T$ to itself. For given any functions $u_1,u_2\in \Theta^s_T,$ we have 
\begin{align}
	\|\Phi(u_1)&(t)-\Phi(u_2)(t)\|_{\Theta_T^s}\\
	&=\frac{1}{\Gamma(\gamma)}\Bigg\|\int_0^te^{(t-\tau)\Delta_{H}}\left(\int_0^\tau(\tau-\tau_1)^{\gamma-1}|u(\tau_1)|^{p}d\tau_1-\int_0^\tau(\tau-\tau_1)^{\gamma-1}|v(\tau_1)|^{p}d\tau_1\right)d\tau\Bigg\|_{L^\infty(H^N)}\\
	&\leq 2^{p-1}C(p)\delta(u_0,f)^{p-1}T^{1+\gamma}\|u_1-u_2\|_{L^\infty((0,T),L^\infty(\mathbbm{H}^N))}\\
	&\leq \|u_1-u_2\|_{L^\infty((0,T),L^\infty(\mathbbm{H}^N))},
\end{align}
for small enough $T>0.$
Finally, an application of the Banach fixed point theorem ensures the existence of a mild solution to \eqref{eq without}. 

Next, we prove the uniqueness part. Let, if possible, $u_1,$ $u_2$ be two mild solutions to \eqref{eq without} for some fixed $T>0$ with $u_1(0)=u_2(0)=u_0.$ Let us define
\begin{align}\label{sup}
	\tau_0\coloneqq \sup\{t	
	\in [0,T] \text{ such that }u_1(\tau)=u_2(\tau)\, \forall \tau\in [0,t]\}.
\end{align}
Assume that $\tau_0\in [0,T).$ In fact, by the continuity of $u_1$ and $u_2$ in time, we have that
\begin{align}
	u_1(\tau_0)=u_2(\tau_0).
\end{align}
Further, we define
\begin{align}
	\widetilde{u}_1(t)\coloneqq u_1(t+\tau_0),\,\widetilde{u}_2(t)\coloneqq u_2(t+\tau_0).
\end{align}
Then, clearly $\widetilde{u}_i\in C([0,T-\tau_0],L^\infty(\mathbbm{H}^N))$ so that it satisfies \eqref{Phii} and $t\in (0,T-\tau_0]$ with 
\begin{align}
	\widetilde{u}_i(0)=u(\tau_0), \text{ for each } i=1,2.
\end{align} 
We further show that for some positive constant $C=C(\tau')<1,$ for some $\tau'\in (0,T-\tau_0),$ we have
\begin{align}\label{ln}
	\sup_{0<t<\tau'}\|\widetilde{u}_1(t)-\widetilde{u}_2(t)\|_{L^\infty(\mathbbm{H}^N)}\leq C\sup_{0<t<\tau'}\|\widetilde{u}_1(t)-\widetilde{u}_2(t)\|_{L^\infty(\mathbbm{H}^N)}.
\end{align}
It would immediately follow that
\begin{align}
	\widetilde{u}_1(t)=\widetilde{u}_2(t)\,\forall \tau\in [0,\tau'].
\end{align}
This provides a contradiction to \eqref{sup} that
\begin{align}
	u_1(t+\tau_0)=u_2(t+\tau_0)\, \forall t\in [0,\tau'].
\end{align}
In view of \eqref{infty}, we obtain the following:
\begin{align}
	\|\widetilde{u}_1&(t)-\widetilde{u}_2(t)\|_{L^\infty(\mathbbm{H}^N)}\\
	&=\frac{1}{\Gamma(\gamma)}\Bigg\|\int_0^te^{(t-\tau)\Delta_{H}}\left(\int_0^{t'}(t'-\tau_1)^{\gamma-1}|u(\tau_1)|^{p}d\tau_1-\int_0^{t'}(t'-\tau_1)^{\gamma-1}|u(\tau_1)|^{p}d\tau_1\right)dt'\Bigg\|_{L^\infty(H^N)}\\
	&\leq 2^{{p-1}}{t'}^{1+\gamma}C(T,\tau_0,\widetilde{u}_1,\widetilde{u}_2)C(p)\delta(u_0,f)^{p-1}\sup_{0<t'<t}\|\widetilde{u}_1(t)-\widetilde{u}_2(t)\|_{L^\infty(\mathbbm{H}^N)},
\end{align}
and for $t$ small enough, we obtain the estimate \eqref{ln}.
This concludes the proof.

(ii) In light of the uniqueness of mild solutions to \eqref{eq without} established above, we deduce the existence of the interval $[0,T_{\text{max}})$ using the well known argument (see \cite{Caze}), where
\begin{align}
	T_{\text{max}}\coloneqq \displaystyle{\sup_{T>0}}\{\eqref{eq without} \text{ admits a solution }u\in C([0,T], L^\infty(\mathbbm{H}^N))\}.
\end{align}
Assume that $T_{\text{max}}$ is finite and that for each $t\in [0,T_{\text{max}}),$ we have
\begin{align}
	\|u(t)\|_{L^\infty(\mathbbm{H}^N)}\leq C,
\end{align} 
for some constant $C>0.$ Fix $t^*\in (T_{\text{max}}/2,T_{\text{max}})$ and for $\widetilde{t} \in (0,T_{\text{max}}),$ let us consider a space
\begin{align}
	\mathcal{M}\coloneqq\{v\in C([0,\widetilde{t}], L^\infty(\mathbbm{H}^N)): \|v\|_{L^\infty([0,\widetilde{t}), L^\infty(\mathbbm{H}^N))}<2\delta(C,\|f\|_{L^\infty(\mathbbm{H}^N)}),v(0)=u(t^*)\},
\end{align}
where \begin{align}
	\delta(C,\|f\|_{L^\infty(\mathbbm{H}^N)})=\max\{C,\|f\|_{L^\infty(\mathbbm{H}^N)}\}.
\end{align} 
Similarly to the arguments followed in the proof of (i) above, we define a map
\begin{align}\label{Phii2}
	\Phi'(v)(t)\coloneqq e^{t\Delta_{H}}u(t^*)+\int_0^te^{(t-\tau)\Delta_{H}}\left(\frac{1}{\Gamma(\gamma)}\int_0^\tau(\tau-\tau_1)^{\gamma-1}|v(\tau_1+t^*)|^{p}d\tau_1+(\tau+t^*)^\alpha f\right)d\tau,
\end{align} 
for $v\in \mathcal{M}$ and $t\in [0,\widetilde{t}].$ Like-wise (i), one can see that $\Phi':\mathcal{M}\To\mathcal{M}$ given by  \eqref{Phii2} is a contraction map. Therefore, we have the existence of a fixed point of $\mathcal{M},$ say $v$ using the Banach fixed point theorem.  Furthermore, let us take $t^*$ such that
\begin{align}
	\widetilde{t}+t^*>T_{\text{max}},
\end{align}
and consider
\begin{align}\label{Ttt}
	\overline{u}(t)\coloneqq\begin{cases}
		u(t), &\text{ for }0\leq t\leq {t}^*\\
		v(t-{t}^*) &\text{ for }{t}^*\leq t\leq t^*+\widetilde{t}.
	\end{cases}
\end{align}
Observe that $\overline{u}\in C([0,t^*+\widetilde{t}], L^\infty(\mathbbm{H}^N)) $ is a solution of \eqref{eq without}. This together with \eqref{Ttt} contradicts the definition of $T_{\text{max}}.$ Therefore, we get
\begin{align}
	\|u(t)\|_{L^\infty(\mathbbm{H}^N)}\To\infty
	\text{ as } t\To T_{\text{max}}.
\end{align}	
This concludes the proof.

(iii) Instead of the space defined by \eqref{e13},
\begin{align}
	\delta(u_0,f)\coloneqq\max\{\|u_0\|_{L^\infty(\mathbbm{H}^N)}, \|f\|_{L^\infty(\mathbbm{H}^N)}\},
\end{align}
and
\begin{align}
	\delta_q(u_0,f)\coloneqq\max\{\|u_0\|_{L^q(\mathbbm{H}^N)}, \|f\|_{L^q(\mathbbm{H}^N)}\},
\end{align} 
consider
\begin{align}
	\widetilde{\Theta}^s_T\coloneqq \bigg\{&v\in C([0,T_{\text{max}}),L^\infty(\mathbbm{H}^N))\cap C([0,T_{\text{max}}),L^q(\mathbbm{H}^N))\\
	&\qquad\qquad\qquad:\|v\|_{L^\infty((0,T_{\text{max}}),L^\infty(\mathbbm{H}^N))}\leq 2\delta(u_0,f),\, \|v\|_{L^\infty((0,T_{\text{max}}),L^q(\mathbbm{H}^N))}\leq 2\delta_q(u_0,f)\bigg\}.
\end{align}
We equip the space $\widetilde{\Theta}^s_T$ with the norm induced by the following distance:
\begin{align}
	d_{q,\infty}(u,v)=\|u-v\|_{L^\infty((0,T_{\text{max}}),L^\infty(\mathbbm{H}^N))}+\|u-v\|_{L^\infty((0,T_{\text{max}}),L^q(\mathbbm{H}^N))},\,u,v\in \widetilde{\Theta}^s_T. 
\end{align}
In light of the inequality
\begin{align}
	\||u(t)|^p\|_{L^q(\mathbbm{H}^N)}\leq \|u(t)\|^{p-1}_{L^\infty(\mathbbm{H}^N)}\|u(t)\|_{L^q(\mathbbm{H}^N)}, 
\end{align}
using the argument similar to (i) gives the existence of a unique solution $u\in \widetilde{\Theta}^s_T.$ Hence, $u\in C([0,T_{\text{max}}),L^\infty(\mathbbm{H}^N))\cap C([0,T_{\text{max}}),L^q(\mathbbm{H}^N)).$ Thus, (iii) follows. This completes the proof. \qed


\section{Finite-time blow-up analysis}\label{Glo ex}
\noindent \textbf{Proof of Theorem \ref{Thm Blowup O}.}
We give a proof by method of contradiction. Let, if possible, $u$ be a global-in-time solution to \eqref{eq 0.1} (equivalently, $v$ be a global-in-time solution to \eqref{eq equi}). For a fixed 
\begin{align}\label{lle}
	l=1+\frac{p+\gamma}{p-1},
\end{align}
we define a function \begin{align}\label{Var}
	\phi(\xi,t)=\phi_1(\xi)\phi_2(t).
\end{align} Here
\begin{align}
	\phi_1(\xi)=\varphi\left(\frac{\|\xi\|_{H}}{T^{\frac{1}{2}}}\right), \, \phi_2(t)=\left(1-\frac{t}{T}\right)^l,\,t\in [0,T],\,T>0.
\end{align}
and $\varphi:\mathbb{R}^+\cup\{0\}\To\mathbb{R}$ is a smooth function defined as
\begin{align}\label{Phip}
	\varphi(r)=\begin{cases}
		1 &\text{ for } r\in [0,1],\\
		\searrow &\text{ for } r\in [1,2],\\
		0  &\text{ for } r\in [2,\infty).
	\end{cases}
\end{align}
Now, by Definition \ref{Weak}, we have
\begin{align}\label{am}
	\\-\int_0^T\int_{\mathbbm{H}^N}	v(\|\xi\|_{H}^{-{2\mu(\lambda)}}\partial_t\phi-L_{H}\phi)d\xi dt &=\int_0^T\int_{\mathbbm{H}^N}\|\xi\|_{H}^{{-(p+1)\mu(\lambda)}}|v|^pI^\gamma_{T^-}\phi d\xi dt\\
	& \qquad+\int_{\mathbbm{H}^N}\|\xi\|_{H}^{{-2\mu(\lambda)}}v_0\phi(0,\cdot) d\xi+\int_0^T\int_{\mathbbm{H}^N}\|\xi\|_{H}^{{-\mu(\lambda)}}t^\alpha f\phi d\xi dt,
\end{align}
Next, consider
\begin{align}
	B^{H}_{2\sqrt T}\coloneqq \{\xi\in\mathbbm{H}^N: \|\xi\|_{H}\leq 2\sqrt T\}.
\end{align}
This infers in light of \eqref{am} that
\begin{align}\label{ka}
	\qquad\quad\int_0^T\int_{B^{H}_{2\sqrt T}}\|\xi\|_{H}^{{-(p+1)\mu(\lambda)}}|v|^pI^\gamma_{T^-}\phi d\xi dt +\int_{B^{H}_{2\sqrt T}}&\|\xi\|_{H}^{{-2\mu(\lambda)}}v_0\phi(0,\cdot) d\xi+\int_0^T\int_{B^{H}_{2\sqrt T}}\|\xi\|_{H}^{{-\mu(\lambda)}}t^\alpha f\phi d\xi dt\\
	&=	-\int_0^T\int_{\mathbbm{H}^N}	v(\|\xi\|_{H}^{-{2\mu(\lambda)}}\partial_t\phi-L_{H}\phi) d\xi dt.
\end{align}
We calculate
\begin{align}
	I_{T^-}^\gamma\phi_2(t)&=\frac{1}{\Gamma(\gamma)}\int_t^T(s-t)^{\gamma-1}\phi_2(s)ds\\
	&=\frac{1}{\Gamma(\gamma)}\int_t^T(s-t)^{\gamma-1}\left(1-\frac{s}{T}\right)^lds.
\end{align}
Using a change of variable $z=\frac{s-t}{T-t},$ we get
\begin{align}\label{IT}
	I_{T^-}^\gamma\phi_2(t)&=\frac{(T-t)^{\gamma-1}}{\Gamma(\gamma)}\frac{(T-t)^l}{T^l}\int_0^1z^{\gamma-1}(1-z)^{l}dz\\
	&=\frac{(T-t)^{\gamma+l}}{\Gamma(\gamma)T^l}B(\gamma,l+1),\,t\in [0,T),
\end{align}
where the Beta function $B$ is defined in Definition \ref{Beta}. Also,
\begin{align}
	\nabla_{H}\phi_1(\xi)&=\varphi'\left(\frac{\|\xi\|_{H}}{T^{\frac{1}{2}}}\right)\nabla_{H}\left(\frac{\|\xi\|_{H}}{T^{\frac{1}{2}}}\right)\\
	&=\frac{1}{T^{\frac{1}{2}}}\varphi'\left(\frac{\|\xi\|_{H}}{T^{\frac{1}{2}}}\right)\nabla_{H}\|\xi\|_H
\end{align}
Let us consider
\begin{align}
	a_1=|v(\xi,t)|\|\xi\|_{H}^{-\mu(\lambda)\left(\frac{p+1}{p}\right)}(I^\gamma_{T^-}\phi(\xi,t))^{\frac{1}{p}},\,\,
	b_1=\|\xi\|_{H}^{-\mu(\lambda)\left(\frac{p-1}{p}\right)}(I^\gamma_{T^-}\phi(\xi,t))^{-\frac{1}{p}}|\partial_t\phi(\xi,t)|,
\end{align}
and
\begin{align}
	a_2=|v(\xi,t)|\|\xi\|_{H}^{-\mu(\lambda)\left(\frac{p+1}{p}\right)}(I^\gamma_{T^-}\phi(\xi,t))^{\frac{1}{p}} ,\,\,
	b_2=\|\xi\|_{H}^{\mu(\lambda)\left(\frac{p+1}{p}\right)}(I^\gamma_{T^-}\phi(\xi,t))^{-\frac{1}{p}}L_H\phi(\xi,t).
\end{align}
For some $\varepsilon>0,$ using the $\varepsilon$-Young inequality, we have 
\begin{align}\label{varyou}
	a_ib_i\leq \varepsilon a_i^p+C(\varepsilon)b_i^{\frac{p}{p-1}}, \text{ for }i=1,2. 
\end{align}
Utilizing \eqref{varyou} in \eqref{ka}, we obtain in view of the fact $\phi_1\in C_c^\infty(\mathbbm{H}^N)$ that
\begin{align}
	\int_0^T&\int_{B^{H}_{2\sqrt T}}\|\xi\|_{H}^{{-(p+1)\mu(\lambda)}}|v|^pI^\gamma_{T^-}\phi  d\xi dt+\int_{B^{H}_{2\sqrt T}}\|\xi\|_{H}^{{-2\mu(\lambda)}}v_0\phi(0,\xi) d\xi +\int_0^T\int_{B^{H}_{2\sqrt T}}\|\xi\|_{H}^{{-\mu(\lambda)}}t^\alpha f d\xi dt\\
	&\leq \varepsilon\int_0^T\int_{B^{H}_{2\sqrt T}}|v|^p\|\xi\|_{H}^{-{\mu(\lambda)({p+1})}}I^\gamma_{T^-}\phi d\xi dt+C(\varepsilon)\int_0^T\int_{B^{H}_{2\sqrt T}}\|\xi\|_{H}^{{-\mu(\lambda)}}(I^\gamma_{T^-}\phi)^{-\frac{1}{p-1}}|\partial_t\phi|^{\frac{p}{p-1}} d\xi dt\\
	&\qquad+\varepsilon\int_0^T\int_{B^{H}_{2\sqrt T}} |v|^p\|\xi\|_{H}^{-\mu(\lambda)({p+1})}I^\gamma_{T^-}\phi d\xi dt\\
	&\qquad+C(\varepsilon)\int_0^T\int_{B^{H}_{2\sqrt T}}\|\xi\|_{H}^{\mu(\lambda)\left(\frac{p+1}{p-1}\right)}(I^\gamma_{T^-}\phi)^{-\frac{1}{p-1}}(L_H\phi)^{\frac{p}{p-1}} d\xi dt.
\end{align}
The above inequality further simplifies to
\begin{align}\label{l1q}
	\\(1-&2\varepsilon)\int_0^T\int_{B^{H}_{2\sqrt T}}\|\xi\|_{H}^{{-(p+1)\mu(\lambda)}}|v|^pI^\gamma_{T^-}\phi d\xi dt+\int_{B^{H}_{2\sqrt T}}\|\xi\|_{H}^{-2\mu(\lambda)}v_0\phi(0,\xi) d\xi+\int_0^T\int_{B^{H}_{2\sqrt T}} \|\xi\|_{H}^{-\mu(\lambda)}t^\alpha f\phi d\xi dt\\
	&\qquad\leq C(\varepsilon)\int_0^T\int_{B^{H}_{2\sqrt T}}\|\xi\|_{H}^{{-\mu(\lambda)}}(I^\gamma_{T^-}\phi)^{-\frac{1}{p-1}}|\partial_t\phi|^{\frac{p}{p-1}}d\xi dt\\
	&\qquad\qquad+C(\varepsilon)\int_0^T\int_{B^{H}_{2\sqrt T}}\|\xi\|_{H}^{\mu(\lambda)\left(\frac{p+1}{p-1}\right)}(I^\gamma_{T^-}\phi)^{-\frac{1}{p-1}}(L_H\phi)^{\frac{p}{p-1}}d\xi dt.
\end{align}
In view of the the intrinsic homogeneity of $\Delta_H,$ $L_{H}$ and $I^\gamma_{T^-},$ the following set of expressions hold:
\begin{align}
	\Delta_{H}(\phi_1(\Phi_\lambda \xi))=\lambda^{2}\Delta_{H}\phi_1(\Phi_\lambda \xi),
\end{align}
\begin{align}
	I^\gamma_{T^-}(\phi_2(\lambda t))=\lambda^{-\gamma}(I^\gamma_{\lambda T^-}\phi_2)(\lambda t),
\end{align}
and
\begin{align}
	\nabla_{H}(\|\Phi_{\lambda}\xi\|_{H}^{-2\mu(\lambda)}\nabla_{H}\phi_1(\Phi_{\lambda}\xi))=\lambda^{2+2\mu(\lambda)}\left(\nabla_{H}(\|\xi\|_{H}^{-2\mu(\lambda)}\nabla_{H}\phi_1)\right)(\Phi_{\lambda}\xi).
\end{align}
These together with the following change of variables 
\begin{align}
	\xi'=\Phi_{T^{-\frac{1}{2}}}\xi \text{ and } t'=T^{-1}t,
\end{align} 
in the R.H.S. of \eqref{l1q} deduce
\begin{align}\label{lep}
	\\&(1-2\varepsilon)\int_0^T\int_{B^{H}_{2\sqrt T}}\|\xi\|_{H}^{{-(p+1)\mu(\lambda)}}|v|^pI^\gamma_{T^-}\phi d\xi dt+\int_{B^{H}_{2\sqrt T}}\|\xi\|_{H}^{-2\mu(\lambda)}v_0\phi(0,\xi)d\xi +\int_0^T\int_{B^{H}_{2\sqrt T}} \|\xi\|_{H}^{-\mu(\lambda)}t^\alpha f\phi d\xi dt\\
	&\leq C(\varepsilon)T^{-\frac{\mu(\lambda)}{2}}T^{\frac{Q}{2}+1}T^{\frac{-\gamma}{p-1}}T^{-\frac{p}{p-1}}\int_0^1\int_{B^{H}_{2}}\|\xi'\|_{H}^{-{\mu(\lambda)}}\big(\phi_1\big(\Phi_{T^{\frac{1}{2}}}\xi'\big)\big)^{{-\frac{1}{p-1}}}(I^\gamma_{TT^-}\phi_2(Tt'))^{-\frac{1}{p-1}}|\partial_t\phi(\Phi_{T^{\frac{1}{2}}}\xi',Tt')|^{\frac{p}{p-1}}d\xi'dt'\\
	&\quad+C(\varepsilon)T^{\frac{(p+1)\mu(\lambda)}{2(p-1)}}T^{\frac{Q}{2}+1}T^{-\frac{\gamma}{p-1}}T^{-\frac{\left(\mu(\lambda)+1\right)p}{p-1}}\times\\
	&\quad \int_0^1\int_{B^{H}_{2}}\|\xi'\|_{H}^{\mu(\lambda)\left(\frac{p+1}{p-1}\right)}\big(\phi_1\big(\Phi_{T^{\frac{1}{2}}}\xi'\big)\big)^{{-\frac{1}{p-1}}}(I^\gamma_{TT^-}\phi_2(Tt'))^{-\frac{1}{p-1}}(\phi_2(Tt'))^{\frac{p}{p-1}}(L_H\phi_1(\Phi_{T^{\frac{1}{2}}}\xi'))^{\frac{p}{p-1}}d\xi'dt'\\
	&=C(\varepsilon)T^{\frac{Q}{2}+1-\frac{\gamma+p}{p-1}-\frac{\mu(\lambda)}{2}}\int_{0}^1\int_{B^{H}_{2}}\|\xi'\|_{H}^{-{\mu(\lambda)}}\big(\phi_1\big(\Phi_{T^{\frac{1}{2}}}\xi'\big)\big)^{\frac{-1}{p-1}}(I^\gamma_{TT^-}\phi_2(Tt'))^{-\frac{1}{p-1}}\phi_1(\Phi_{T^{\frac{1}{2}}}\xi')^{\frac{p}{p-1}}|\partial_t\phi_2(Tt')|^{\frac{p}{p-1}}d\xi'dt'\\
	&\quad+C(\varepsilon)T^{\frac{Q}{2}+1-\frac{\gamma+p}{p-1}-\frac{\mu(\lambda)}{2}}\times\\
	&\qquad \int_0^1\int_{B^{H}_{2}}\|\xi'\|_{H}^{\mu(\lambda)\left(\frac{p+1}{p-1}\right)}\big(\phi_1\big(\Phi_{T^{\frac{1}{2}}}\xi'\big)\big)^{{-\frac{1}{p-1}}}(I^\gamma_{TT^-}\phi_2(Tt'))^{-\frac{1}{p-1}}(\phi_2(Tt'))^{\frac{p}{p-1}}(L_H\phi_1(\Phi_{T^{\frac{1}{2}}}\xi'))^{\frac{p}{p-1}}d\xi'dt'.
\end{align}
Since $\phi_1\in C_c^\infty(\mathbbm{H}^N),$ and derivative-support of $\phi_1$ lies in a compact annulus $1\leq \|\xi'\|_{H}\leq 2.$ Also, recall that $\mu<Q.$ 
These facts further yield the following set of estimates
\begin{align}
	\int_{0}^1\int_{B^{H}_{2}}\|\xi'\|_{H}^{-{\mu(\lambda)}}\big(\phi_1\big(\Phi_{T^{\frac{1}{2}}}\xi'\big)\big)^{\frac{-1}{p-1}}(I^\gamma_{TT^-}\phi_2(Tt'))^{-\frac{1}{p-1}}\phi_1(\Phi_{T^{\frac{1}{2}}}\xi')^{\frac{p}{p-1}}|\partial_t\phi_2(Tt')|^{\frac{p}{p-1}}d\xi'dt'&\leq C_1,\\
	\int_0^1\int_{B^{H}_{2}}\|\xi'\|_{H}^{\mu(\lambda)\left(\frac{p+1}{p-1}\right)}\big(\phi_1\big(\Phi_{T^{\frac{1}{2}}}\xi'\big)\big)^{{-\frac{1}{p-1}}}(I^\gamma_{TT^-}\phi_2(Tt'))^{-\frac{1}{p-1}}(\phi_2(Tt'))^{\frac{p}{p-1}}(L_H\phi_1(\Phi_{T^{\frac{1}{2}}}\xi'))^{\frac{p}{p-1}}d\xi'dt'&\leq C_2,
\end{align} 
where $C_1$ and $C_2$ are constants, which are independent of $T.$
Furthermore, taking $\varepsilon$ arbitrarily small together with the following relation
\begin{align}\label{s3}
	\int_{B^{H}_{2\sqrt T}}\|\xi\|_{H}^{-2\mu(\lambda)}&v_0\phi(0,\xi)d\xi+\int_0^T\int_{\mathbbm{H}^N}\|\xi\|_H^{-\mu(\lambda)}t^\alpha f\phi d\xi dt\\
	&=\int_{B^{H}_{2\sqrt T}}\|\xi\|_{H}^{-2\mu(\lambda)}v_0\phi_1(\xi)d\xi+T^{(\alpha+1)}\int_{\mathbbm{H}^N}\|\xi\|_H^{-\mu(\lambda)}f\phi_1(\xi) d\xi\int_0^1t'^\alpha \phi_2(Tt')dt'\\
	&\geq   \int_{B^{H}_{2\sqrt T}}\|\xi\|_{H}^{-2\mu(\lambda)}v_0\phi_1(\xi)d\xi+CT^{(\alpha+1)}\int_{\mathbbm{H}^N}\|\xi\|_H^{-\mu(\lambda)}f\phi_1 d\xi,
\end{align}
in the expression \eqref{lep} yields
\begin{align}
	\int_{B^{H}_{2\sqrt T}}\|\xi\|_{H}^{-2\mu(\lambda)}v_0\phi_1d\xi+CT^{(\alpha+1)}\int_{\mathbbm{H}^N}\|\xi\|_H^{-\mu(\lambda)}f\phi_1d\xi
	&\leq C(\varepsilon)T^{\frac{Q}{2}+1-\frac{\gamma+p}{p-1}-\frac{\mu(\lambda)}{2}}C_1+C_1C(\varepsilon)T^{\frac{Q}{2}+1-\frac{\gamma+p}{p-1}-\frac{\mu(\lambda)}{2}} C_2\\
	&\leq CT^{\frac{Q}{2}+1-\frac{\gamma+p}{p-1}-\frac{\mu(\lambda)}{2}},
\end{align}
for some positive constant $C$ independent of $T.$ This further simplifies to
\begin{align}\label{s1}
	T^{-(\alpha+1)}\int_{B^{H}_{2\sqrt T}}\|\xi\|_{H}^{-2\mu(\lambda)}v_0\phi_1d\xi+\int_{B^{H}_{2\sqrt T}}\|\xi\|_H^{-\mu(\lambda)}f\phi_1 d\xi
	\leq CT^{\frac{Q}{2}-\frac{\gamma+p}{p-1}-\frac{\mu(\lambda)}{2}-\alpha}.
\end{align}
Now, observe that 
\begin{align}\label{s2}
	\frac{Q}{2}-\frac{\gamma+p}{p-1}-\frac{\mu(\lambda)}{2}-\alpha
	&=\frac{Q(p-1)-2(\gamma+p)-(p-1)\mu(\lambda)-2\alpha(p-1)}{2(p-1)}\\
	&=\frac{p(Q-2-\mu(\lambda)-2\alpha)-(Q+2\gamma-\mu(\lambda)-2\alpha)}{2(p-1)}\\
	&< 0,
\end{align}
since \begin{align}
	1< p<\frac{Q+2\gamma-\mu(\lambda)-2\alpha}{Q-2-\mu(\lambda)-2\alpha}.
\end{align}
Finally, in light of \eqref{s2}, and since $\alpha\in (-1,0),$ taking the limit $T\to\infty$ in \eqref{s1} provides
\begin{align}
	\int_{\mathbbm{H}^N}\|\xi\|_H^{-\mu(\lambda)}f(\xi)d\xi\leq 0,
\end{align}
which is a contradiction to our hypothesis that $\int_{\mathbbm{H}^N}\|\xi\|_H^{-\mu(\lambda)}f(\xi)d\xi>0.$ This proves our claim.\qed\\

\noindent \textbf{Proof of Theorem \ref{Thm Blowup O2}.}		
Let, if possible, $u$ be a global-in-time solution to \eqref{eq 0.1}. We follow the similar arguments as in Theorem \ref{Thm Blowup O} with minor change in the test functions.
Consider 
\begin{align}
	\phi(\xi,t)=\phi_1(\xi)\phi_2(t).
\end{align} 
Here
\begin{align}
	\phi_1(\xi)=\varphi\left(\epsilon{\|\xi\|_{H}}\right), \, \phi_2(t)=\left(1-\frac{t}{T}\right)^l
\end{align}
for large enough positive constant $T$ and a fixed sufficiently small number $\epsilon>0.$ We recall that $\varphi$ is the same function defined in \eqref{Phip}.

Mimicking the similar lines as in the proof of Theorem \ref{Thm Blowup O}, with the above changes in the definition of Test function, we obtain the following inequality in place of \eqref{lep}:
\begin{align}
	(1-2\varepsilon)&\int_0^T\int_{B^{H}_{2\sqrt T}}\|\xi\|_{H}^{{-(p+1)\mu(\lambda)}}|v(t,\xi)|^pI^\gamma_{T^-}\phi(t,\xi) d\xi dt+\int_{B^{H}_{2\sqrt T}}\|\xi\|_{H}^{-2\mu(\lambda)}v_0\phi(0,\xi)d\xi\\
	\qquad\qquad&+\int_0^T\int_{B^{H}_{2\sqrt T}} \|\xi\|_{H}^{-\mu(\lambda)}t^\alpha f(\xi)\phi(t,\xi) d\xi dt\leq CT^{1-\frac{\gamma+p}{p-1}}+CT^{1-\frac{\gamma}{p-1}}.
\end{align}
Further, following the similar steps as used after \eqref{lep} above, we get
\begin{align}\label{f2l}
	T^{-(\alpha+1)}\int_{B^{H}_{2\sqrt T}}\|\xi\|_{H}^{-2\mu(\lambda)}v_0\phi_1d\xi+\int_{\mathbbm{H}^N}\|\xi\|_{H}^{-\mu(\lambda)}f\phi_1d\xi&\leq CT^{1-\frac{\gamma+p}{p-1}-1-\alpha
	}+CT^{1-\frac{\gamma}{p-1}-1-\alpha}\\
	&\leq C\left(T^{-\frac{\gamma+p}{p-1}-\alpha
	}+T^{-\frac{\gamma}{p-1}-\alpha}\right).
\end{align}
Finally, in view of the fact that $\gamma\geq 0$ and $\alpha>0,$ we get a contradiction on letting $T\To\infty.$ This completes the proof.\qed

\noindent\textbf{Proof of Theorem \ref{Thm Blowup O3}.} Following the similar lines of proof as in Theorem \ref{Thm Blowup O2} till \eqref{f2l}, we have
\begin{align}
	T^{-(\alpha+1)}\int_{B^{H}_{2\sqrt T}}\|\xi\|_{H}^{-2\mu(\lambda)}v_0\phi_1d\xi+\int_{\mathbbm{H}^N}\|\xi\|_{H}^{-\mu(\lambda)}f\phi_1d\xi&\leq CT^{1-\frac{\gamma+p}{p-1}-1-\alpha
	}+CT^{1-\frac{\gamma}{p-1}-1-\alpha}\\
	&\leq C\left(T^{-\frac{\gamma+p}{p-1}-\alpha
	}+T^{-\frac{\gamma}{p-1}-\alpha}\right).
\end{align}
Finally, in view of the fact that $\gamma> 0$ and $\alpha\geq 0,$ we get a contradiction on letting $T\To\infty.$ This completes the proof. 
\qed

\section{Lifespan estimate}\label{Life s}
Building upon the result, Theorem \ref{Thm Blowup O}, we address the question of how rapidly blow-up occurs in the subcritical regime.
We denote the maximal existence time by 
\begin{align}
	T_\varepsilon\coloneqq\sup_{T>0}\{\eqref{eq 0.1} \text{ admits a solution in }\mathbbm{H}^N\times [0,T)\},
\end{align}
and prove the following estimate:
\begin{thm}\label{life}
	Let $\gamma\in [0,1),$ $\alpha\in (-1,0),$ and $p\in \bigg(1,\frac{Q+2\gamma-\mu(\lambda)-2\alpha}{Q-2-\mu(\lambda)-2\alpha}\bigg).$ Assume that $u_0\in L^\infty(\mathbbm{H}^N)$ with $u_0\geq 0,$ $0<\int_{\mathbbm{H}^N}\|\xi\|_H^{-\mu(\lambda)}f(\xi)d\xi<\infty$ and 
	\begin{align}\label{fxi}
		f(\xi)\geq \varepsilon\|\xi\|_{H}^{-2\beta}, \text{ for }\|\xi\|_{H}\geq 1, \beta\in \left(\frac{Q}{2}-\frac{\mu(\lambda)}{2},\frac{\gamma+p}{p-1}+\alpha\right) \text{ and }\varepsilon>0. 
	\end{align} 
	Then for any local-in-time solution $u_\varepsilon$ to \eqref{eq 0.1}, the following estimate on $T_\varepsilon$ is valid: 
	\begin{align}
		T_\varepsilon\leq C\varepsilon^{\big({\beta-\frac{\gamma+p}{p-1}-\alpha}\big)^{-1}}, 
	\end{align}
	for some constant $C>0$ independent of $T_\varepsilon.$	
\end{thm}
\begin{proof}
	In proving this result, we make use of the test functions introduced in the proof of Theorem \ref{Thm Blowup O}. Let $u_\varepsilon$ be local-in-time solution to \eqref{eq 0.1} with its maximal time $T_\varepsilon$ and also define $v_\varepsilon(\xi,t)\coloneqq\|\xi\|_{H}^{{\mu(\lambda)}}u_\varepsilon(\xi,t).$
	Now, for a fixed 
	\begin{align}
		l=1+\frac{p+\gamma}{p-1},
	\end{align}
	we define a function \begin{align}
		\phi(\xi,t)=\phi_1(\xi)\phi_2(t).
	\end{align} 
	Here
	\begin{align}
		\phi_1(\xi)=\varphi\left(\frac{\|\xi\|_{H}}{T_{\varepsilon}^{\frac{1}{2}}}\right), \, \phi_2(t)=\left(1-\frac{t}{T_\varepsilon}\right)^l,
	\end{align}
	and $\varphi$ is the same function appearing in \eqref{Phip}.
	Now, proceeding along similar lines as in the proof of Theorem \ref{Thm Blowup O}, and incorporating the following change of variables 
	\begin{align}\label{ch var}
		\xi'=\Phi_{T_\varepsilon^{-\frac{1}{2}}}\xi \text{ and } t'=T_\varepsilon^{-1}t,
	\end{align}
	provides
	\begin{align}\label{lepo}
		\\(1&-2\varepsilon)\int_0^{T_\varepsilon}\int_{B^{H}_{2\sqrt T_\varepsilon}}\|\xi\|_{H}^{{-(p+1)\mu(\lambda)}}|v_\varepsilon|^pI^\gamma_{T_\varepsilon^-}\phi d\xi dt+\int_{B^{H}_{2\sqrt T_\varepsilon}}\|\xi\|_{H}^{-2\mu(\lambda)}v_0\phi(0,\xi)d\xi+\int_0^{T_\varepsilon}\int_{B^{H}_{2\sqrt T_\varepsilon}} \|\xi\|_{H}^{-\mu(\lambda)}t^\alpha f\phi d\xi dt\\&\leq C(\varepsilon){T_\varepsilon}^{\frac{Q}{2}+1-\frac{\gamma+p}{p-1}-\frac{\mu(\lambda)}{2}}\int_{0}^1\int_{B^{H}_{2}}\|\xi'\|_{H}^{-{\mu(\lambda)}}\big(\phi_1\big(\Phi_{T_\varepsilon^{\frac{1}{2}}}\xi'\big)\big)^{\frac{-1}{p-1}}(I^\gamma_{T_\varepsilon T^-}\phi_2(T_\varepsilon t'))^{-\frac{1}{p-1}}\phi_1(\Phi_{T_\varepsilon^{\frac{1}{2}}}\xi')^{\frac{p}{p-1}}|\partial_t\phi_2(T_\varepsilon t')|^{\frac{p}{p-1}}d\xi'dt'\\
		&\quad+C_1C(\varepsilon)T_\varepsilon^{\frac{Q}{2}+1-\frac{\gamma+p}{p-1}-\frac{\mu(\lambda)}{2}}\times \\
		&\qquad \int_0^1\int_{B^{H}_{2}}\|\xi'\|_{H}^{\mu(\lambda)\left(\frac{p+1}{p-1}\right)}\big(\phi_1\big(\Phi_{T_\varepsilon^{\frac{1}{2}}}\xi'\big)\big)^{{-\frac{1}{p-1}}}(I^\gamma_{T_\varepsilon T^-}\phi_2(T_\varepsilon t'))^{-\frac{1}{p-1}}(\phi_2(T_\varepsilon t'))^{\frac{p}{p-1}}(L_H\phi_1(\Phi_{T_\varepsilon^{\frac{1}{2}}}\xi'))^{\frac{p}{p-1}}d\xi'dt',
	\end{align}
	for some constant $C_1$ independent of $T_\varepsilon,$ since $\phi_1\in C_c^\infty(\mathbbm{H}^N).$ Using the facts that $u_0\geq 0$ and $\epsilon$ is arbitrarily small together with the following relation
	\begin{align}\label{s3}
		\int_0^{T_\varepsilon}\int_{\mathbbm{H}^N}\|\xi\|_{H}^{-\mu(\lambda)}t^\alpha f\phi d\xi dt&=T_\varepsilon^{(\alpha+1)}\int_{\mathbbm{H}^N}\|\xi\|_{H}^{-\mu(\lambda)}f(\xi)\phi_1(\xi)d\xi\int_0^1t'^\alpha \phi_2(T_\varepsilon t')dt'\\
		&\geq   CT_\varepsilon^{(\alpha+1)}\int_{\mathbbm{H}^N}\|\xi\|_{H}^{-\mu(\lambda)}f(\xi)\phi_1(\xi) d\xi,
	\end{align}
	in the expression \eqref{lepo} yields
	\begin{align}
		CT_\varepsilon^{(\alpha+1)}\int_{\mathbbm{H}^N}\|\xi\|_H^{-\mu(\lambda)}f(\xi)\phi_1(\xi)d\xi
		&\leq CT_\varepsilon^{\frac{Q}{2}+1-\frac{\gamma+p}{p-1}-\frac{\mu(\lambda)}{2}},
	\end{align}
	for some positive constant $C$ independent of $T_\varepsilon.$ This simplifies to
	\begin{align}\label{s10}
		\int_{\mathbbm{H}^N}\|\xi\|_{H}^{-\mu(\lambda)}f(\xi)\phi_1(\xi)d\xi
		\leq CT_\varepsilon^{\frac{Q}{2}-\frac{\gamma+p}{p-1}-\alpha-\frac{\mu(\lambda)}{2}}.
	\end{align}
	Further, using the estimate \eqref{fxi} together with \eqref{ch var}, we obtain
	\begin{align}\label{mt0}
		\int_{\mathbbm{H}^N}\|\xi\|_{H}^{-\mu(\lambda)}f(\xi)\phi_1(\xi)d\xi&=T_\varepsilon^{\frac{Q}{2}}\int_{\mathbbm{H}^N}\|\Phi_{T_\varepsilon^{\frac{1}{2}}}\xi\|_{H}^{-\mu(\lambda)}f\bigg(\Phi_{T_\varepsilon^{\frac{1}{2}}}\xi'\bigg)\phi_1\bigg(\Phi_{T_\varepsilon^{\frac{1}{2}}}\xi'\bigg)d\xi'\\
		&\geq \varepsilon T_\varepsilon^{\frac{Q}{2}-\beta-\frac{\mu(\lambda)}{2}}\int_{\mathbbm{H}^N\setminus B^{H}_{\frac{1}{\sqrt{T}_\varepsilon}}}\|\xi'\|_{H}^{-\mu(\lambda)}\|\xi'\|_{H}^{-2\beta}\phi_1\bigg(\Phi_{T_\varepsilon^{\frac{1}{2}}}\xi'\bigg)d\xi'\\
		&\geq \varepsilon T_\varepsilon^{\frac{Q}{2}-\beta-\frac{\mu(\lambda)}{2} }\int_{\mathbbm{H}^N\setminus B^{H}_{\frac{1}{\sqrt{T}}}}\|\xi'\|_{H}^{-2\beta-\mu(\lambda)}\phi_1\bigg(\Phi_{T_\varepsilon^{\frac{1}{2}}}\xi'\bigg)d\xi',
	\end{align}
	for all $T\leq T_\varepsilon.$ Here, $\Phi$ denotes the dilation given by \eqref{dil}.
	Finally, utilizing \eqref{mt0} in \eqref{s10} yields
	\begin{align}
		\varepsilon T_\varepsilon^{\frac{Q}{2}-\beta-\frac{\mu(\lambda)}{2}}\int_{\mathbbm{H}^N\setminus B^{H}_{\frac{1}{\sqrt{T}}}}\|\xi'\|_{H}^{-2\beta-\mu(\lambda)}\phi_1\bigg(\Phi_{T_\varepsilon^{\frac{1}{2}}}\xi'\bigg)d\xi'\leq CT_\varepsilon^{\frac{Q}{2}-\frac{\gamma+p}{p-1}-\alpha-\frac{\mu(\lambda)}{2}}.
	\end{align}
Furthermore, in view of the fact that $2\beta>Q-\mu$ (see \eqref{fxi}), we get
	\begin{align}
		\varepsilon&\leq CT_\varepsilon^{\frac{Q}{2}-\frac{\gamma+p}{p-1}-\alpha-\frac{Q}{2}+\beta}\\
		&=CT_\varepsilon^{-\frac{\gamma+p}{p-1}-\alpha+\beta}.
	\end{align}
	This deduces
	\begin{align}
		T_\varepsilon\leq C\varepsilon^{\big({\beta-\frac{\gamma+p}{p-1}-\alpha}\big)^{-1}},
	\end{align}
	for some positive constant $C$ independent of $T_\varepsilon.$ This completes the proof.
\end{proof}
\section{Discussion}
We have characterized the finite-time blow-up and local-in-time existence of solutions to semilinear evolution equations with Hardy-type potentials and memory terms on the Heisenberg group. In particular, we identified critical exponents for both instantaneous and finite-time blow-up, highlighting the interplay between the singular potential, temporal nonlocality, and the external source term.

As mentioned earlier, the present model is novel because it simultaneously incorporates
a singular Hardy potential, nonlocal-in-time memory term, and a time-dependent source. The results are new even in the Euclidean setting, including the homogeneous case $f=0$ in equation \eqref{eq 0.1}. 

A major limitation of the current work is that the question of global-in-time existence for general initial data,  when both the Hardy potential and the memory term are present, remains open. In addition, studying a nonlocal-in-space counterpart of the present model, obtained by replacing $-\Delta_H$ with the fractional sub-Laplacian $(-\Delta_H)^s,$ is completely open. One of the main challenges in such an extension is that the sharp constant in the fractional Hardy inequality on the Heisenberg group is still unknown. The analytical tools available at present do not seem sufficient to settle these questions. Nevertheless, in the absence of the Hardy potential and memory term, the heat equations involving fractional sub-Laplacian have been very recently investigated in \cite{Bui, Oza, Oza Suragan 1}. We expect that further advances in sub-Riemannian energy methods or refined a priori estimates may offer new insights into these challenging questions.
\stoptoc
\section{Declarations} 
\noindent

\subsection*{Funding } This research was funded by Nazarbayev University under Collaborative Research Program Grant 20122022CRP1601.

\subsection*{Availability of data and materials}     
Not Applicable.

\subsection*{Conflicts of interests/Competing interests}
There are no conflict of interests of any type.
\resumetoc

\end{document}